\documentclass[a4paper,10pt,leqno]{amsart}
\usepackage[total={6in,9in},
top=1in, left=1in, right=1in, bottom=1in]{geometry}
\usepackage{amsmath}
\usepackage[italian, english]{babel}
\usepackage{amsfonts}
\usepackage{amssymb}
\usepackage{dsfont}
\usepackage{mathrsfs}
\usepackage{graphicx}
\usepackage{setspace}
\usepackage{fancyhdr}
\usepackage{amsthm}
\usepackage{empheq}
\usepackage{cases}
\usepackage[all]{xy}
\usepackage{stmaryrd}
\newcommand{\varrg}{(M, g)}

\newcommand{\erre}{\mathds{R}}

\newcommand{\cinf}{C^{\infty}(M)}

\newcommand{\ricc}{\operatorname{Ric}}
\newcommand{\diver}{\operatorname{div}}
\newcommand{\hess}{\operatorname{Hess}}

\newcommand{\ra}{\rightarrow}

\newcommand{\set}[1]{{\left\{#1\right\}}}               
\newcommand{\pa}[1]{{\left(#1\right)}}                  
\newcommand{\sq}[1]{{\left[#1\right]}}                  
\newcommand{\abs}[1]{{\left|#1\right|}}                 

\newcommand{\riemanng}[1]{\pa{#1,g}}                      

\renewcommand{\hat}[1]{\widehat{#1}}
\renewcommand{\tilde}[1]{\widetilde{#1}}

\newtheorem{theorem}{\textbf{Theorem}}[section]
\newtheorem{lemma}[theorem]{\textbf{Lemma}}
\newtheorem{proposition}[theorem]{\textbf{Proposition}}
\newtheorem{cor}[theorem]{\textbf{Corollary}}
\newtheorem{defi}[theorem]{\textbf{Definition}}
\theoremstyle{remark}
\newtheorem{rem}[theorem]{\textbf{Remark}}

\numberwithin{equation}{section}

\title[Einstein-type manifolds]
{On the geometry of gradient Einstein-type manifolds}

\date{\today} 

\keywords{}

\subjclass[2010]{53C20; 53C25, 53A55}

\begin{document}

\maketitle

\begin{center}
\textsc{\textmd{Giovanni Catino\footnote{Politecnico di Milano,
Italy. Email: giovanni.catino@polimi.it. Partially supported by GNAMPA, section ``Calcolo delle Variazioni, Teoria del Controllo e Ottimizzazione'', and GNAMPA, project ``Equazioni differenziali con invarianze in Analisi Globale''.}, Paolo
Mastrolia\footnote{Universit\`{a} degli Studi di Milano, Italy.
Email: paolo.mastrolia@gmail.com. Partially supported by FSE,
Regione Lombardia.}, Dario D. Monticelli\footnote{Universit\`{a}
degli Studi di Milano, Italy. Email: dario.monticelli@gmail.com.
Partially supported by GNAMPA, section ``Equazioni differenziali e
sistemi dinamici'', and GNAMPA, project ``Equazioni differenziali con invarianze in Analisi Globale''.} and Marco
Rigoli\footnote{Universit\`{a} degli Studi di Milano, Italy. Email:
marco.rigoli@unimi.it.}, }}
\end{center}
\begin{abstract}
In this paper we introduce the notion of Einstein-type structure on a Riemannian manifold $\varrg$, unifying various particular cases recently studied in the literature, such as gradient Ricci solitons,  Yamabe solitons and quasi-Einstein manifolds. We show that these general structures can be locally classified when the Bach tensor is null. In particular, we extend a recent result of Cao and Chen \cite{CaoChen}.
\end{abstract}


\section{Introduction and main results}

In the last years there has been an increasing interest in the
study of  Riemannian manifolds endowed with metrics
satisfying some structural equations, possibly involving curvature
and some  globally defined vector fields. These objects naturally
arise in several different frameworks; two of the most important and
well studied examples are \emph{Einstein manifolds} (see e.g.
\cite{Jensen}, \cite{Besse}, \cite{WangZiller1}, \cite{WangZiller2}, \cite{MasMonRig_Curvature}) and \emph{Ricci solitons} (see e.g.
\cite{hamilton}, \cite{Ivey}, \cite{perelman1}, \cite{ELnM},
\cite{NiWallach},  \cite{ZHZhang}, \cite{Naber}, \cite{petwylie3},
\cite{CaoZhou}, \cite{HDCaoChen_Steady}, \cite{CatMantEv},
\cite{PRiS}, \cite{XDCaoWangZhang_ShrinkingRS},
\cite{Catino_pinched}, \cite{CaoChen}, \cite{CaoCatinoChenMantMazz},
\cite{MasRigRim}, \cite{brendle} and references therein). Other
examples are, for instance, \emph{Ricci almost solitons}
(\cite{PRRS_Almost}), \emph{Yamabe solitons}
(\cite{DiCerboDisconzi}, \cite{MaCheng},\cite{DaskaSesum},
\cite{HDCaoSunZhang}), \emph{Yamabe quasi-solitons}
(\cite{HuangLi_quasiYamabe}, \cite{Wang_quasiYamabe}),
\emph{conformal gradient solitons} (\cite{tashiro},
\cite{catmantmazz}), \emph{quasi-Einstein manifolds} (\cite{kimkim},
\cite{CaseShuWei}, \cite{CatMantMazzRim_quasiEinstein},
\cite{HePetersenWylie}, \cite{MastroliaRimoldi_triviality}),
\emph{$\rho$-Einstein solitons} (\cite{CatinoMazz_rhoEinstein1},
\cite{CatinoMazzMongodi_rhoEinstein2}).


%
%

In this paper we study Riemannian manifolds satisfying a general
structural condition that includes all the aforementioned examples as
particular cases.

Towards this aim we consider a smooth, connected Riemannian manifold $\varrg$ of
dimension $m\geq 3$, and we denote with $\ricc$ and $S$ the
corresponding \emph{Ricci tensor} and \emph{scalar curvature},
respectively (see the next section for the details). We
denote with $\hess(f)$ the Hessian of a function $f\in \cinf$ and
with $\mathcal{L}_Xg$ the Lie derivative of the metric $g$ in the
direction of the vector field $X$. We introduce the following

\begin{defi}\label{DE_GETM}
We say that $\varrg$ is an \emph{Einstein-type manifold} (or, equivalently, that $\varrg$ supports an \emph{Einstein-type structure}) if there
exist $X\in\mathfrak{X}(M)$ and $\lambda \in \cinf$ such that
\begin{equation}\label{Eq_ETS_generic_global}
  \alpha\ricc +\frac{\beta}{2}\mathcal{L}_Xg+\mu X^\flat\otimes X^\flat = \pa{\rho S+\lambda}g,
\end{equation}
for some constants $\alpha, \beta, \mu, \rho \in \erre$, with $(\alpha, \beta, \mu)\neq (0,0,0)$.
If $X=\nabla f$ for some $f\in\cinf$, we say that $\varrg$ is a
\emph{gradient Einstein-type manifold}. Accordingly equation
\eqref{Eq_ETS_generic_global} becomes
\begin{equation}\label{Eq_ETS_global}
  \alpha\ricc +\beta\hess(f)+\mu df\otimes df = \pa{\rho S+\lambda}g,
\end{equation}
for some $\alpha, \beta, \mu, \rho \in \erre$.
\end{defi}
\noindent Here $X^\flat$ denotes the $1$-form metrically dual to $X$.

%

In the present paper we focus our analysis on the gradient case,
postponing the general case to a subsequent work.

Leaving aside the case $\beta=0$ that will be addressed separately, see Proposition \ref{PR_betanullo} below, we say that the gradient Einstein-type manifold $\varrg$ is  \emph{nondegenerate} if $\beta \neq 0$ and
$\beta^2\neq (m-2)\alpha\mu$; otherwise, that is if $\beta \neq 0$ and
$\beta^2= (m-2)\alpha\mu$ we have a \emph{degenerate} gradient Einstein-type manifold. Note that, in this last case, necessarily $\alpha$ and $\mu$ are not null. The above terminology is justified by the next observation:
\begin{equation}\label{121}
\begin{array}{c}
  \varrg \, \text{is conformally Einstein if and only if} \\ \text{for some } \alpha, \beta, \mu \neq 0, \,  \varrg \text{ is a degenerate, gradient Einstein-type manifold.}
  \end{array}
\end{equation}
For the proof and for the notion of conformally Einstein manifold see Section \ref{sec_2} below.

In case $f$ is constant we say that the Einstein-type structure is \emph{trivial}. Note that, since $m\geq 3$, in this case $\varrg$ is Einstein. However, the converse is generally false; indeed, if $\varrg$ is Einstein, then for some constant $\Lambda \in \erre$ we have $\ricc = \Lambda g$ and inserting into \eqref{Eq_ETS_global} we obtain
\[
\beta \hess(f) +\mu df\otimes df = \pa{\rho S + \lambda-\Lambda\alpha}g.
\]
Thus, if $\rho \neq 0$, $\varrg$ is a Yamabe quasi-soliton and $f$ is not necessarily constant.

\noindent We will also deal with the case $\alpha=0$ separately, see Theorem \ref{TH_alfa0_quasiYamabe} below. We explicitly remark that, as a simple consequence of \eqref{Eq_ETS_global}, $\alpha$ and $\beta$ cannot both be equal to zero.

As we have already noted, the class of manifolds satisfying Definition
\ref{DE_GETM} gives rise to the  previously quoted examples
by specifying, in general not in a unique way, the values of the
parameters and possibly the function $\lambda$. In particular we
have:
\begin{enumerate}
  \item Einstein manifolds: $\pa{\alpha, \beta, \mu,  \rho} = \pa{1, 0, 0, \frac{1}{m}}, \lambda=0$ (or, equivalently for $m\geq 3$, $\rho=0$ and $\lambda = \frac{S}{m}$);
  \item Ricci solitons: $\pa{\alpha, \beta, \mu,  \rho} = \pa{1, 1, 0, 0}, \lambda \in \erre$;
    \item Ricci almost solitons: $\pa{\alpha, \beta, \mu,  \rho} = \pa{1, 1, 0, 0}, \lambda \in \cinf$;
    \item Yamabe solitons: $\pa{\alpha, \beta, \mu,  \rho} = \pa{0, 1, 0, 1}, \lambda\in \erre$;
       \item Yamabe quasi-solitons: $\pa{\alpha, \beta, \mu,  \rho} = \pa{0, 1, -\frac{1}{k}, 1}, k\in \erre\setminus\set{0}, \lambda \in \erre$;
           \item conformal gradient solitons: $\pa{\alpha, \beta, \mu,  \rho} = \pa{0, 1, 0, 0}, \lambda \in \cinf$;
               \item quasi-Einstein manifolds:  $\pa{\alpha, \beta, \mu,  \rho} = \pa{1, 1, -\frac{1}{k}, 0},\, \lambda \in \erre$, $k \neq 0$;
                \item $\rho$-Einstein solitons:  $\pa{\alpha, \beta, \mu,  \rho} = \pa{1, 1, 0, \rho}, \,\,\rho\neq 0,  \lambda \in \erre$.
         \end{enumerate}

Of course one may wonder about the existence of Einstein-type structures. We know from the literature positive answers to the various examples the we mentioned earlier. For the general case we will consider three different necessary conditions; the first two are the general integrability conditions \eqref{firstGeneralIntCond} and \eqref{secondGeneralIntCond} contained in Theorem \ref{TH_integrabilityConditions} below. The third comes from the simple observation that, tracing equation \eqref{Eq_ETS_global} and defining $u=e^{\frac{\mu}{\beta}f}$, the existence of a gradient Einstein-type structure on $\varrg$ yields the existence of a positive solution of
\[
L u = \Delta u -\frac{\mu}{\beta}\sq{m\lambda+\pa{m\rho-\alpha}S}u = 0,
\]
so that, by a well-known spectral result (see for instance Fischer-Colbrie-Schoen \cite{FCS}, or Moss-Piepenbrink \cite{MPiepen}), the operator $L$ is stable, or, in other words, the spectral radius of $L$, $\lambda^L_1(M)$, is nonnegative.
In Section \ref{sec_2.5} below we shall give some simple conditions on the function $\frac{\mu}{\beta}\sq{m\lambda+\pa{m\rho-\alpha}S}$ that prevent this possibility, so that the corresponding Einstein-type structure cannot exist.

As it appears from Definition \ref{DE_GETM}, the fact that $\varrg$ is an Einstein-type manifold can be interpreted as a prescribed condition on the Ricci tensor of $g$ (see for instance the nice survey \cite{jpb}), that is, on the ``trace part'' of the Riemann tensor. Thus, it is reasonable to expect classification and rigidity results for these structures only assuming further conditions on the traceless part of the Riemann tensor, i.e. on the Weyl tensor. Indeed, most of the aforementioned papers pursue this direction, for instance, assuming that $\varrg$ is locally conformally flat or has harmonic Weyl tensor. In the spirit of the recent work of H.-D. Cao and Q. Chen \cite{CaoChen}, we study the class of gradient Einstein-type manifolds with vanishing Bach tensor along the integral curves of $f$. We note that this condition is weaker than local conformal flatness (see Section \ref{sec_2}).

It turns out that, as in the case of gradient Ricci solitons (see \cite{HDCaoChen_Steady}, \cite{CaoChen} and \cite{CaoCatinoChenMantMazz}), the leading actor is a three tensor, $D$,  that plays a fundamental role in relating the Einstein-type structure to the geometry of the underlying manifold. $D$ naturally appears when writing the first two integrability conditions for the structure defining the differential system \eqref{Eq_ETS_global}. Quite unexpectedly, the constant $\rho$ and the function $\lambda$ have no influence on this relation.

Our main purpose is to give  local characterizations of complete, noncompact, nondegenerate gradient Einstein-type manifolds. Denoting with $B$ the Bach tensor of $\varrg$ (see Section \ref{sec_2}), our first result is

\begin{theorem}\label{TH_Main}
  Let $\varrg$ be a complete, noncompact, nondegenerate gradient Einstein-type manifold of dimension $m\geq 3$.
  If $B\pa{\nabla f, \cdot}=0$ and $f$ is a proper function, then, in a neighbourhood of every regular level set of $f$, the manifold $\varrg$ is locally a warped product with $(m-1)$-dimensional Einstein fibers.
\end{theorem}

In dimension four we  improve this result, obtaining

\begin{cor} \label{COR_Main}
  Let $(M^{4},g)$ be a complete, noncompact nondegenerate gradient Einstein-type manifold of dimension four.
 If $B\pa{\nabla f, \cdot}=0$ and $f$ is a proper function, then, in a neighbourhood of every  regular level set of $f$, the manifold $\varrg$ is locally a warped product with three-dimensional fibers of constant curvature. In particular, $(M^{4}, g)$ is locally conformally flat.
\end{cor}

As we will show in Section \ref{sec_8}, the properness assumption is satisfied by some important subclasses of Einstein-type manifolds, under quite natural geometric assumptions. As a consequence, in the case of gradient Ricci solitons, we recover a local version of the results in \cite{CaoChen} and \cite{CaoCatinoChenMantMazz}, while, in the cases of $\rho$-Einstein solitons and Ricci almost solitons, we prove two new classification theorems (see Theorem \ref{TH_appl1} and \ref{TH_appl2}).

In the special case $\alpha =0$ (which includes Yamabe solitons,
Yamabe quasi-solitons and conformal gradient solitons) we give a version of Theorem \ref{TH_Main} in the following local result that provides a very precise description of the metric in this situation. Note that Theorem \ref{TH_alfa0_quasiYamabe} and Corollary \ref{COR_alfa0_quasiYamabe} also apply to the compact case.

\begin{theorem}\label{TH_alfa0_quasiYamabe} Let $\varrg$ be a complete gradient Einstein-type manifold of dimension $m\geq 3$ with $\alpha=0$. Then, in a neighbourhood of every regular level set of $f$, the manifold $\varrg$ is locally a warped product with $(m-1)$-dimensional fibers. More precisely, every regular level set $\Sigma$ of $f$ admits a maximal open neighborhood $U\subset M^m$ on which $f$ only depends on the signed distance $r$ to the hypersurface $\Sigma$. In addition, the potential function $f$ can be chosen in such a way that the metric $g$ takes the form
\begin{equation}\label{warped metric}
g \, = \, dr \otimes dr \,+ \,\left(\frac{f'(r)}{f'(0)}e^{\mu f(r)}\right)^{2}\, g^{\Sigma} \quad {\hbox{on $U$}} ,
\end{equation}
where $g^{\Sigma}$ is the metric induced by $g$ on $\Sigma$. As a consequence, $f$ has at most two critical points on $M^m$ and we have the following cases:
\begin{itemize}
\item[(1)] If $f$ has no critical points, then $\varrg$ is globally conformally equivalent to a direct product $I\times N^{m-1}$ of some interval $I=(t_{*},t^{*})\subseteq \mathbb{R}$ with a $(m-1)$-dimensional complete Riemannian manifold $(N^{m-1},g^{N})$. More precisely, the metric takes the form
$$
g \, = \, u^{2}(t)\, \big(dt^{2}+g^{N}\big) \, ,
$$
where $u:(t_{*},t^{*})\rightarrow \mathbb{R}$ is some positive smooth function.
\smallskip

\item[(2)] If $f$ has only one critical point $O\in M^m$, then $\varrg$ is globally conformally equivalent to the interior of a Euclidean ball of radius $t^{*}\in(0,+\infty]$. More precisely, on $M^{m}\setminus~\{O\}$, the metric takes the form
$$
g \, = \, v^{2}(t)\, \big(dt^{2}+t^{2}g^{\mathbb{S}^{m-1}}\big) \,,
$$
where $v:(0,t^{*})\rightarrow \mathbb{R}$ is some positive smooth function and $\mathbb{S}^{m-1}$ denotes the standard unit sphere of dimension $m-1$. In particular $\varrg$ is complete, noncompact and rotationally symmetric.
\smallskip
%
%
\item[(3)] If the function $f$ has two critical points $N,S \in M^m$, then $\varrg$ is globally conformally equivalent to $\mathbb{S}^{m}$.  More precisely, on $M^{m}\setminus \{N,S\}$, the metric takes the form
$$
g \, = \, w^{2}(t)\, \big(dt^{2}+\sin^{2}(t)\,g^{\mathbb{S}^{m-1}}\big) \,,
$$
where $w:(0,\pi)\rightarrow \mathbb{R}$ is some smooth positive function. In particular $\varrg$ is compact and rotationally symmetric.
\end{itemize}
\end{theorem}

In this case, we can obtain a stronger global result, just assuming nonnegativity of the Ricci curvature; namely we have the following

\begin{cor}\label{COR_alfa0_quasiYamabe}
Any nontrivial, complete, gradient Einstein type manifold with $\alpha=0$ and nonnegative Ricci curvature  is either rotationally symmetric or it is isometric to a Riemannian product $\mathbb{R}\times N^{m-1}$, where $N^{m-1}$ is an $(m-1)$-dimensional Riemannian manifold with nonnegative Ricci
curvature.
\end{cor}

This result covers the cases of Yamabe solitons~\cite{HDCaoSunZhang} and conformal gradient solitons~\cite{catmantmazz}. Concerning Yamabe quasi-solitons, Corollary~\ref{COR_alfa0_quasiYamabe} improves the results in~\cite{HuangLi_quasiYamabe}. In particular, this shows that most of the assumptions in \cite[Theorem 1.1]{HuangLi_quasiYamabe} are not necessary.

The paper is organized as follows. In Section \ref{sec_2} we recall some useful definitions and properties of various geometric tensors and fix our conventions and notation. Next, in Section \ref{sec_2.5} we deal with nonexistence of gradient Einstein-type structures, both in the degenerate and in the nondegenerate case,  and we give some sufficient conditions for $\lambda^L_1(M)<0$.
In Section \ref{sec_3}  we collect some useful commutations relations for covariant derivatives of functions and tensors. In Section \ref{sec_4} we treat the special case of gradient Einstein-type manifolds with $\alpha=0$ proving Theorem \ref{TH_alfa0_quasiYamabe} and Corollary \ref{COR_alfa0_quasiYamabe}. In Section \ref{sec_5} we prove the two aforementioned integrability conditions that follow directly from the Einstein-type structures. In Section \ref{sec_6} we compute the squared norm of the tensor $D$ in terms of $D$ itself, the Bach tensor $B$ and the potential function $f$. In Section \ref{sec_7} we relate the tensor $D$ to the geometry of the regular level sets of the potential function $f$. Finally, in Section \ref{sec_8} we prove Theorem \ref{TH_Main} and Corollary \ref{COR_Main}, and we give some geometric applications in the special cases of gradient Ricci solitons, $\rho$-Einstein solitons and Ricci almost solitons.

\

\section{Definitions and notation}
\label{sec_2}

In this section we recall some useful definitions and properties of various geometric tensors and fix our conventions and notation (see also \cite{MasRigSet}).

To perform computations, we freely use the method of the moving
frame referring to a local orthonormal coframe of the
$m$-dimensional Riemannian manifold $\varrg$. We fix the index range
$1\leq i, j, \ldots \leq m$ and recall that the Einstein summation
convention will be in force throughout.

We denote with $\operatorname{R}$ the \emph{Riemann curvature
tensor} (of type $\pa{1, 3}$) associated to the metric $g$, and with
$\ricc$ and $S$ the corresponding \emph{Ricci tensor} and
\emph{scalar curvature}, respectively. The components of the $(0,
4)$-versions of the Riemann tensor and of the \emph{Weyl tensor}
$\operatorname{W}$ are related by the formula:
\begin{equation}\label{Riemann_Weyl}
  R_{ijkt} = W_{ijkt} + \frac{1}{m-2}\pa{R_{ik}\delta_{jt}-R_{it}\delta_{jk}+R_{jt}\delta_{ik}-R_{jk}\delta_{it}}-\frac{S}{(m-1)(m-2)}\pa{\delta_{ik}\delta_{jt}-\delta_{it}\delta_{jk}}
\end{equation}
and they satisfy the symmetry relations
\begin{eqnarray}
&&R_{ijkt} = -R_{jikt} = -R_{ijtk} = R_{ktij},\\
&&W_{ijkt} = -W_{jikt} =-W_{ijtk} = W_{ktij}.
\end{eqnarray}
A computation shows that the Weyl tensor is also totally trace-free. According to this convention the (components of the) Ricci tensor
and the scalar curvature are respectively given by $R_{ij} = R_{itjt} = R_{titj}$ and $S = R_{tt}$. The \emph{Schouten tensor} $\mathrm{A}$ is defined by
\begin{equation}\label{def_Schouten}
  \mathrm{A} = \ricc -\frac{S}{2(m-1)}g .
\end{equation}
Tracing we have $\operatorname{tr}(\mathrm{A}) = A_{tt} = \frac{(m-2)}{2(m-1)}S$.
\begin{rem}
  Some authors adopt a different convention and define the Schouten tensor as $\frac{1}{m-2}A$.
\end{rem}

We note that, in terms of the Schouten tensor and of the Weyl
tensor, the Riemann curvature tensor can be expressed in the form
\begin{equation}\label{decompRiemSchouten}
  \operatorname{R} = \textrm{W} + \frac{1}{m-2}\textrm{A}\owedge g,
\end{equation}
where $\owedge$ is the Kulkarni-Nomizu product; in components,
\begin{equation}\label{Riemann_Weyl_Schouten}
  R_{ijkt} = W_{ijkt} + \frac{1}{m-2}\pa{A_{ik}\delta_{jt}-A_{it}\delta_{jk}+A_{jt}\delta_{ik}-A_{jk}\delta_{it}}.
\end{equation}

Next we introduce the \emph{Cotton tensor} $C$ as the obstruction to
the commutativity of the covariant derivative of the Schouten
tensor, that is
 \begin{equation}\label{def_Cotton_comp}
   C_{ijk} = A_{ij, k} - A_{ik, j} = R_{ij, k} - R_{ik, j} - \frac{1}{2(m-1)}\pa{S_k\delta_{ij}-S_j\delta_{ik}}.
 \end{equation}
 We also recall that the Cotton tensor, for $m\geq 4$, can be defined as one of the possible divergences of the Weyl
 tensor; precisely
 \begin{equation}\label{def_Cotton_comp_Weyl}
 C_{ijk}=\pa{\frac{m-2}{m-3}}W_{tikj, t}=-\pa{\frac{m-2}{m-3}}W_{tijk, t}.
 \end{equation}
 A computation shows that the two definitions (for $m\geq4$) coincide (see again \cite{MasRigSet}).

\begin{rem}
It is worth to recall that the Cotton tensor is skew-symmetric in
the second and third indices (i.e. $C_{ijk}=-C_{ikj}$) and totally
trace-free (i.e. $C_{iik}=C_{iki}=C_{kii}=0$).
\end{rem}

We are now ready to define the \emph{Bach tensor} $B$, originally
introduced by Bach in \cite{Bach} in the study of conformal
relativity. Its components are
 \begin{equation}\label{def_Bach_comp}
   B_{ij} =  \frac{1}{m-2}\pa{C_{jik, k}+R_{kt}W_{ikjt}},
 \end{equation}
that, in case $m\geq 4$, by \eqref{def_Cotton_comp_Weyl} can be
alternatively written as
 \begin{equation}
  B_{ij} = \frac{1}{m-3}W_{ikjt, tk} + \frac{1}{m-2}R_{kt}W_{ikjt}.
\end{equation}
Note that if $\varrg$ is either locally conformally flat (i.e. $C=0$
if $m=3$ or $W=0$ if $m\geq 4$) or Einstein, then $B=0$. A
computation  shows that the Bach tensor is symmetric (i.e.
$B_{ij}=B_{ji}$) and evidently trace-free (i.e. $B_{ii}=0$). As a
consequence we observe that we can write
  \[
   B_{ij}= \frac{1}{m-2}\pa{C_{ijk, k}+R_{kl}W_{ikjl}}.
  \]


We recall that
\begin{defi}
  The manifold $\varrg$ is \emph{conformally Einstein} if its metric $g$
  can be pointwise conformally deformed to an Einstein metric
  $\tilde g$.
\end{defi}
We observe that, if $\tilde g = e^{2a\varphi}g$, for some $\varphi
\in \cinf$ and some constant $a\in \erre$, then its Ricci tensor
$\widetilde{\ricc}$ is related to that of $g$ by the well-known
formula (see for instance \cite{MasRigSet})
\begin{equation}\label{2.11}
  \widetilde{\ricc} = \ricc
  -(m-2)a\hess(\varphi)+(m-2)a^2d\varphi\otimes\varphi-\sq{(m-2)a^2\abs{\nabla
  \varphi}^2+a\Delta\varphi}g.
\end{equation}
Here the various operators (and for their precise definitions see
Section \ref{sec_3}) are defined with respect to the metric $g$.

Now we can easily prove  statement \eqref{121}; indeed, suppose that
$\beta \neq 0$ and $\beta^2=(m-2)\alpha\mu$, that is, the
Einstein-type structure is degenerate. Tracing \eqref{Eq_ETS_global}
we obtain
\begin{equation}\label{2.12}
  \frac{1}{\alpha}\pa{\rho S+\lambda} = \frac{1}{m}\pa{S+\frac{\beta}{\alpha}\Delta f+\frac{\mu}{\alpha}\abs{\nabla
  f}^2}.
\end{equation}
Choose $\varphi=f$ and $a=-\frac{\beta}{(m-2)\alpha}$ in
\eqref{2.11} to obtain
\begin{equation}\label{2.13}
\widetilde{\ricc} =
\frac{1}{\alpha}\sq{\frac{\beta^2}{(m-2)\alpha}-\mu}df \otimes df
+\frac{1}{\alpha}\pa{\rho S+\lambda}g
+\frac{\beta}{(m-2)\alpha}\pa{\Delta
f-\frac{\beta}{\alpha}\abs{\nabla f}^2}g.
\end{equation}
Inserting \eqref{2.12} into \eqref{2.13} yields
\begin{equation*}
\widetilde{\ricc} =
\frac{1}{\alpha}\sq{\frac{\beta^2}{(m-2)\alpha}-\mu}df \otimes df
+\frac{1}{m}\sq{S+2\frac{\beta}{\alpha}\frac{m-1}{m-2}\Delta
f-\frac{\mu}{\alpha}\pa{m-1}\abs{\nabla f}^2}g.
\end{equation*}
Hence, since $\beta^2=(m-2)\alpha\mu$,
\begin{equation}\label{2.13.1}
\widetilde{\ricc} =
\frac{1}{m}\sq{S+2\frac{\beta}{\alpha}\frac{m-1}{m-2}\Delta
f-\frac{\mu}{\alpha}\pa{m-1}\abs{\nabla f}^2}g,
\end{equation}
that is, $\tilde{g} = e^{-\frac{2\beta}{(m-2)\alpha}f}g$ is an
Einstein metric (this was also obtained in Theorem 1.159 of
\cite{Besse}).

Viceversa, suppose that $\tilde{g}= e^{2af}g$, $a\neq 0$, is an
Einstein metric, so that, for some $\Lambda\in\erre$,
$\widetilde{\ricc}=\Lambda\tilde{g}$. From \eqref{2.11}
\begin{equation}\label{2.14}
\ricc-(m-2)a\hess(f)+(m-2)a^2df\otimes df=\sq{\Lambda
e^{2af}+(m-2)a^2\abs{\nabla f}^2+a\Delta f}g.
\end{equation}
Tracing we get
\begin{equation*}
\frac{S}{m-1}=\sq{(m-2)a^2\abs{\nabla f}^2+a\Delta f}+a\Delta
f+\frac{m}{m-1}\Lambda e^{2af}.
\end{equation*}
Thus, inserting into \eqref{2.14},
\begin{equation*}
\ricc-(m-2)a\hess(f)+(m-2)a^2df\otimes df=\pa{\frac{S}{m-1}-a\Delta
f-\frac{\Lambda}{m-1}e^{2af}}g.
\end{equation*}
We choose $\alpha=1$, $\beta=-(m-2)a$, $\mu=(m-2)a^2$,
$\rho=\frac{1}{m-1}$ and $\lambda(x)=-a\Delta
f-\frac{\Lambda}{m-1}e^{2af}$. We note that $\beta\neq0$ and
\begin{equation*}
\beta^2=(m-2)^2a^2=(m-2)\alpha\mu,
\end{equation*}
so that the above choice of $\alpha,\beta,\mu,\rho$ and $\lambda$
yields a degenerate Einstein-type structure.

\section{Nonexistence of gradient Einstein-type structures}\label{sec_2.5}

In this section we comment on the nonexistence of gradient
Einstein-type structures on $\varrg$.
From now on we fix an origin $o\in M$ and let $r(x)=\operatorname{dist}\pa{x, o}$. We set $B_r$ and $\partial B_r$ to denote, respectively, the geodesic ball of radius $r$ centered at $o$ and its boundary.

We begin with considering the
\emph{degenerate case}. In this situation $\beta\neq0$ and $\beta^2=(m-2)\alpha\mu$; in particular $\alpha, \mu \neq 0$. Multiplying equation \eqref{Eq_ETS_global} by $\frac{1}{m-2}\frac{\mu}{\alpha}$ and setting $h=\frac{\mu}{\beta}f$, using the relation $\beta^2=(m-2)\alpha\mu$ we immediately obtain
 \[
 \frac{\mu}{m-2}\ricc +\mu\hess(h)+\mu dh\otimes dh = \pa{\frac{\mu}{\beta}}^2\pa{\rho S+\lambda}g,
 \]
that is, another degenerate gradient Einstein-type structure. Using \eqref{2.13.1} with our new constants and with $h$ replacing $f$ we deduce the existence of a constant $\Lambda \in \erre$ such that
\[
\Lambda e^{-2h} = S +2(m-1)\Delta h-(m-1)(m-2)\abs{\nabla h}^2.
\]

We set $u=e^{-\frac{m-2}{2}h}$ so that, using the above, $u$ becomes a positive solution of the Yamabe equation

%
%
\begin{equation}\label{2.5.2}
  4\frac{m-1}{m-2}\Delta u -S(x)u+\Lambda u^{\frac{m+2}{m-2}}=0.
\end{equation}
Hence, every time \eqref{2.5.2} has no positive solution, we can
conclude that $\varrg$ has no degenerate gradient Einstein-type
structure. Nonexistence for \eqref{2.5.2} heavily depends on the
sign of $\Lambda$; indeed, let us consider first the case $\Lambda \geq
0$. Thus $u$ satisfies
\[
4\frac{m-1}{m-2}\Delta u -S(x)u \leq 0, \quad u>0 \quad \text{on }
M.
\]
By \cite{FCS}, if $\mathfrak{L}= \Delta  -S(x)\frac{m-2}{4(m-1)}$, then
$\lambda^\mathfrak{L}_1(M)\geq 0$. Hence, in this case, every time we can
guarantee that $\lambda^\mathfrak{L}_1(M)< 0$, there do not exist positive
solutions of \eqref{2.5.2} on $M$. We will give some sufficient
conditions for this at the end of the section.

For the case $\Lambda <0$ the situation is more involved. We recall that with our
choices
\[
u= e^{-\frac{m-2}{2}h}=e^{-\frac{m-2}{2}\frac{\mu}{\beta}f},
\]
so that $u\in L^2(M)$ if and only if $e^{-\pa{m-2}\frac{\mu}{\beta}f} \in L^1(M)$.
Applying Proposition 3.1 of \cite{MasRigSet} we have that for
$\Lambda<0$ there are no gradient, degerate, Einstein-type
structures with $e^{-\pa{m-2}\frac{\mu}{\beta}f} \in L^1(M)$, provided that
$\lambda^\mathfrak{L}_1(M)\geq 0$, $\mathfrak{L}$ as above. The request on the
integrability of $e^{-\pa{m-2}\frac{\mu}{\beta}f}$ can be replaced by
\[
f(x)\ra +\infty \quad\text{as } r(x)\ra +\infty,
\]
provided $\lambda_1^\mathfrak{L}\pa{\operatorname{supp}S_-}>0$, see Theorem
3.12 of \cite{MasRigSet}. Note that, since $\operatorname{supp}S_-$ is a closed set
we need to extend the definition of $\lambda^\mathfrak{L}_1$ to this case. For
a generic bounded subset $D$ of $M$ we set
\[
\lambda_1^\mathfrak{L}(D)=\sup \lambda^\mathfrak{L}_1(\Omega),
\]
where the supremum is taken over all open, bounded sets with smooth
boundary $\Omega$ such that $D\subset\Omega$. Note that, by
definition, if $D=\emptyset$ then $\lambda^\mathfrak{L}_1(D)=+\infty$. Finally,
if $D$ is an unbounded subset of $M$, we define
\[
\lambda_1^\mathfrak{L}(D)=\inf \lambda_1^\mathfrak{L}(D\cap\Sigma),
\]
where the infimum is taken over all bounded open sets $\Sigma$ with
smooth boundary. Observe that, since $\lambda_1^\mathfrak{L}(B_r)\sim
\frac{C}{r^2}$ for some constant $C>0$ as $r\ra+\infty$
(see e.g. \cite{ChBookEigen}) and $B_r$ is a geodesic ball centered at
$p\in M$, the condition $\lambda_1^\mathfrak{L}(\operatorname{supp}S_-)>0$
means that the set $\operatorname{supp}S_-$ is small in a suitable
spectral sense.

Again, using Theorem 5.12 of \cite{MasRigSet}, there are no gradient
degenerate Einstein-type structures on $\varrg$ with
\[
f(x)\ra-\infty\quad\text{as } r(x)\ra +\infty,
\]
for which
\[
\sup_M S_-(x)<+\infty
\]
and
\[
\liminf_{r\ra+\infty} \frac{\log\operatorname{vol}\pa{B_r}}{r^2}<+\infty,
\]
where $\operatorname{vol}\pa{B_r}$ denotes the volume of the geodesic ball $B_r$.
The above discussion also shows the important role played by the
sign of the first eigenvalue of the Dirichlet problem for the
operator $\mathfrak{L}$.

\vspace{0.3cm}

We now analyze the existence for a \emph{nondegenerate gradient Einstein
structure}. As remarked in the introduction, letting
$L=\Delta-\frac{\mu}{\beta}\sq{m\lambda(x)-(m\rho-\alpha)S(x)}$ we have nonexistence
every time $\lambda_1^L(M)<0$.
We let
\[
\frac{\mu}{\beta}\sq{m\bar{\lambda}(r)+\pa{m\rho-\alpha}\bar{S}(r)}=\frac{\mu}{\beta}\frac{1}{\operatorname{vol}\pa{\partial B_r}}\sq{m\int_{\partial B_r}\lambda(x)+\pa{m\rho-\alpha}\int_{\partial B_r}S(x)},
\]
i.e. the radialization of the zeroth-order term.
Note that, given any sufficiently regular function $q(x)$, by the co-area formula
\[
\int_0^R\bar{q}(s)\,\operatorname{vol}\pa{\partial B_s}\,ds = \int_{B_R}q(x).
\]
This fact and the Rayleigh characterization of the first eigenvalue of the Dirichlet problem on the ball $B_R$ justify assumptions on the radialization of the zeroth-order term rather than on the term itself. To simplify the writing we set $v(r)=\operatorname{vol}\pa{\partial B_r}$ and let $\hat{v}(r)$ satisfy $\hat{v}\in L^\infty_{loc}\pa{[0, +\infty)}$, $\frac{1}{\hat{v}}\in L^\infty_{loc}\pa{(0, +\infty)}$, $0\leq v\leq \hat{v}$ on $[0, +\infty)$. We suppose
\[
\frac{1}{\hat{v}}\in L^1\pa{+\infty}
\]
and we define the critical curve associated to $\hat{v}$, $\chi_{\hat{v}}$, by setting
\[
\chi_{\hat{v}}(r) = \set{2\hat{v}(r)\int_r^{+\infty}\frac{ds}{\hat{v}(s)}}^{-2}.
\]
By using Theorem 6.15 in \cite{BMR_Osc} we give the following sufficient condition for the instability of $L$. Assume that
\begin{equation}\label{2.5.3}
  \bar{q}(r)=\frac{\mu}{\beta}\pa{m\bar{\lambda}(r)+\pa{m\rho-\alpha}\bar{S}(r)}\leq 0, \quad \bar{q}\not\equiv 0,
\end{equation}
that $v(r)$ and $\hat{v}(r)$ are as above and that
\begin{equation}\label{2.5.4}
  \limsup_{r\ra +\infty}\int_R^r\pa{\sqrt{\abs{\bar{q}(s)}}-\sqrt{\chi_{\hat{v}}(s)}}\,ds = +\infty
\end{equation}
for some $R\gg 1$. Then $L$ is unstable (in fact, $L$ has infinite index).

We can even prove that $\lambda^L_1(M)<0$ under a less restrictive condition, but in order to avoid technicalities we adopt \eqref{2.5.4}. Indeed, it is not difficult to simplify \eqref{2.5.4} in case we give an explicit upper bound for $v(r)$. For instance, if $\hat{v}(r)=\zeta r^\sigma$, that is
\[
\operatorname{vol}\pa{\partial B_r} \leq \zeta r^\sigma
\]
for $r \gg 1$ and some constants $\zeta>0$, $\sigma>1$, \eqref{2.5.4} becomes
\begin{equation}\label{2.5.5}
   \limsup_{r\ra +\infty}\set{\int_R^r\sqrt{\abs{\bar{q}(s)}}\,ds -\frac{\sigma-1}{2}\log r}= +\infty,
\end{equation}
while for an exponential bound
\[
\operatorname{vol}\pa{\partial B_r} \leq \zeta r^\theta e^{a r^\sigma \log^\tau r}
\]
for $r \gg 1$ and some constants $\zeta, a, \sigma>0$, $\tau\geq 0$, $\theta\in\erre$, \eqref{2.5.4} is equivalent to
\begin{equation}\label{2.5.6}
   \limsup_{r\ra +\infty}\set{\int_R^r\sqrt{\abs{\bar{q}(s)}}\,ds -\frac{a}{2}r^\sigma\log^\tau r-\frac{\sigma+\theta-1}{2}\log r-\frac{\tau}{2}\log\log r}= +\infty.
\end{equation}
As a final remark we observe that condition \eqref{2.5.3} can be relaxed. We refer the interested reader to sections 6.6 and 6.7 in Chapter 6 of \cite{BMR_Osc}.

\section{Some basics on moving frames and commutation rules}
\label{sec_3}

In this section we collect some useful commutation relations for
covariant derivatives  of functions and tensors  that will be used
in the rest of the paper.

Let $\riemanng{M}$ be a Riemannian manifold of dimension $m\geq 3$.
For the sake of completeness (see \cite{MasRigSet} for details)  we
recall that, having fixed a (local) orthonormal coframe
$\set{\theta^i}$,  with dual frame $\set{e_i}$, then the
corresponding \emph{Levi-Civita connection forms}
$\set{\theta^i_j}$  are the $1$-forms uniquely defined by the
requirements
\begin{align}
  &d\theta^i = -\theta^i_j \wedge
  \theta^j \quad \text{(first structure equations)}, \label{1_firstStructureEq} \\
  &\theta^i_j + \theta^j_i = 0. \label{1_skewsymmConnForm}
\end{align}
The \emph{curvature forms} $\set{\Theta^i_j}$ associated to the
connection are the $2$-forms defined via the \emph{second structure
equations}
\begin{equation}\label{1_secondStructureEq}
  {d\theta^i_j = -\theta^i_k \wedge \theta^k_j + \Theta^i_j.}
\end{equation}
They are skew-symmetric (i.e. $\Theta^i_j + \Theta^j_i = 0$) and they can be written as
\begin{equation}\label{1_def_forme_curvatura}
  \Theta^i_j = \frac{1}{2}R^i_{jkt}\theta^k \wedge \theta^t = \sum_{k<t}R^i_{jkt}\theta^k \wedge \theta^t,
\end{equation}
where $R^i_{jkt}$ are precisely the coefficients of the ($(1,
3)$-version of the) Riemann curvature tensor.

 The \emph{covariant derivative of a vector field} $X \in \mathfrak{X}(M)$ is defined
 by
\[
\nabla X = (dX^i + X^j\theta^i_j)\otimes e_i=X^i_k\theta^k \otimes e_i,
\]
while the \emph{covariant derivative of a $1$-form} $\omega$ is
defined by
\[
\nabla \omega = (d\omega_i-w_j\theta^j_i)\otimes \theta^i= \omega_{ik}\theta^k \otimes \theta^i.
\]
The \emph{divergence} of the vector field $X\in\mathfrak{X}(M)$  is the trace of the endomorphism $(\nabla X)^{\sharp} : TM \ra TM$, that is,
\begin{equation}\label{1_divergence}
  \operatorname{div}X = \operatorname{tr}\pa{\nabla X}^{\sharp} = g\pa{\nabla_{e_i}X, e_i} = X_i^i.
\end{equation}
For a smooth function $f$ we can write
\begin{equation}\label{DifferentialComponents}
df = f_i\theta^i,
\end{equation}
for some smooth coefficients $f_i\in\cinf$. The \emph{Hessian} of $f$, $\hess(f)$, is the $(0, 2)$-tensor defined as
\begin{equation}
  \hess(f) = \nabla df = f_{ij}\theta^j\otimes\theta^i,
\end{equation}
with
\begin{equation}\label{HessianComponents}
f_{ij}\theta^j = df_i - f_t\theta_i^t.
\end{equation}
Note that (see Lemma \ref{LemmaCommRulesFunctions} below)
\[
f_{ij}=f_{ji}.
\]

The \emph{Laplacian} of $f$, $\Delta f$,  is the trace of the
Hessian, in other words
\[
\Delta f = \operatorname{tr}(\hess(f)) = f_{ii}.
\]

The moving frame formalism reveals extremely useful in determining
the commutation rules of geometric tensors (see again \cite{MasRigSet} for
details). Some of them will be essential in our computations.

\begin{lemma}\label{LemmaCommRulesFunctions} If $f\in C^3(M)$ then:
\begin{align}
  f_{ij} &= f_{ji}; \label{SecondDerivFunction}\\ f_{ijk} &= f_{jik}; \label{CovDerivSecondDerivFct}\\ f_{ijk} &= f_{ikj}+f_tR_{tijk}; \label{ThirdDerivFunctionRiem}
  \\f_{ijk} &= f_{ikj}+f_tW_{tijk}+\frac{1}{m-2}\pa{f_tR_{tj}\delta_{ik}-f_tR_{tk}\delta_{ij}+f_jR_{ik}-f_kR_{ij}}\label{ThirdDerivFunctionWeyl}\\\nonumber &-\frac{S}{(m-1)(m-2)}\pa{f_j\delta_{ik}-f_k\delta_{ij}}; \\ f_{ijk} &= f_{ikj}+f_tW_{tijk}+\frac{1}{m-2}\pa{f_tA_{tj}\delta_{ik}-f_tA_{tk}\delta_{ij}+f_jA_{ik}-f_kA_{ij}};\label{commutatioThirdDerFunctWeilSchouten}
  \end{align}
In particular, tracing  \eqref{ThirdDerivFunctionRiem}  we deduce
\begin{align}
  f_{itt} &= f_{tti}+f_tR_{ti}. \label{TracedThirdDerivFunctionRicci}
\end{align}
\end{lemma}

\begin{proof}
Let $df = f_i\theta^i$. Differentiating and using the structure equations we get
\begin{align*}
0&=df_i \wedge \theta^i + f_i d\theta^i = (f_{ij}\theta^j +
f_k\theta^k_i)\wedge\theta^i - f_i\theta^i_k \wedge \theta^k =
f_{ij}\theta^j \wedge \theta^i= \frac 12 \pa{f_{ij}-f_{ji}}\theta^j \wedge \theta^i,
\end{align*}
thus
\[
0= \sum_{1 \leq j <i \leq m}(f_{ij}-f_{ji})\theta^j \wedge \theta^i;
\]
since $\set{\theta^j \wedge \theta^i}$  $\pa{1 \leq j < i \leq m}$
is a basis for the $2$-forms we get equation
\eqref{SecondDerivFunction}. Equation \eqref{CovDerivSecondDerivFct}
follows taking the covariant derivative of
\eqref{SecondDerivFunction}. As for \eqref{ThirdDerivFunctionRiem},
by definition of covariant derivative we have
\begin{equation}\label{1_derivate_terze}
  f_{ijk}\theta^k = df_{ij} - f_{kj}\theta^k_i - f_{ik}\theta^k_j.
\end{equation}

Differentiating equation \eqref{HessianComponents} and using the structure
equations we get
\begin{align*}
  df_{ik}\wedge\theta^k - f_{ij}\theta^j_k \wedge\theta^k &= - df_t
\wedge \theta^t_i + f_k\theta^k_t \wedge \theta^t_i - f_k\Theta^k_i
= \\ &=-(f_{tk}\theta^k + f_k\theta^k_t) \wedge\theta^t_i +
f_k\theta^k_t \wedge \theta^t_i - \frac{1}{2}f_kR^k_{ijt} \theta^j
\wedge \theta^t,
\end{align*}
thus
\[
(df_{ik} - f_{tk}\theta^t_i - f_{it}\theta^t_k)\wedge \theta^k
=-\frac{1}{2}f_t R^t_{ijk}\theta^j \wedge \theta^k,
\]
and, by \eqref{1_derivate_terze},
\[
f_{ikj}\theta^j \wedge \theta^k = -\frac{1}{2} f_t R^t_{ijk}
\theta^j \wedge \theta^k.
\]
Skew-symmetrizing we get
\[
\frac{1}{2}(f_{ikj}-f_{ijk})\theta^j \wedge \theta^k =
-\frac{1}{2}f_tR^t_{ijk}\theta^j \wedge\theta^k,
\]
that is, \eqref{ThirdDerivFunctionRiem}. Equations
\eqref{ThirdDerivFunctionWeyl} and
\eqref{commutatioThirdDerFunctWeilSchouten} follow easily from
\eqref{ThirdDerivFunctionRiem}, using the definitions of the Weyl
tensor and of the Schouten tensor (see Section 2).
\end{proof}

For the Riemann curvature tensor we recall the classical Bianchi identities, that in our formalism become
\begin{align}
  &R_{ijkt}+R_{itjk}+R_{iktj}=0  \quad \text{(First Bianchi Identities)};\label{FirstBianchiRiem}\\ &R_{ijkt, l}+R_{ijlk, t}+R_{ijtl, k}=0  \quad \text{(Second Bianchi Identities)}. \label{SecondBianchiRiem}
\end{align}
For the second derivatives of $\operatorname{R}$ we have
\begin{lemma}\label{LemmaSTRiemann}
\begin{align}
   &R_{ijkt, lr}-R_{ijkt, rl} = R_{sjkt}R_{silr}+R_{iskt}R_{sjlr}+R_{ijst}R_{sklr}+R_{ijks}R_{stlr}. \label{SecondDerivRiem}
\end{align}
\end{lemma}
\begin{proof}
 By definition of covariant derivative we have
 \begin{equation}\label{firstCovDerivRiemComp}
   R_{ijkt, l}\theta^l = d R_{ijkt}-R_{ljkt}\theta^l_i-R_{ilkt}\theta^l_j-R_{ijlt}\theta^l_k-R_{ijkl}\theta^l_t
 \end{equation}
 and
  \begin{equation}\label{secondCovDerivRiemComp}
   R_{ijkt, lr}\theta^r = d R_{ijkt, l}-R_{ljkt, l}\theta^r_i-R_{irkt, l}\theta^r_j-R_{ijrt, l}\theta^r_k-R_{ijkr, l}\theta^r_t-R_{ijkt, r}\theta^r_l.
 \end{equation}
 Differentiating equation \eqref{firstCovDerivRiemComp} and using the first structure equations we get
 \begin{align}
   d R_{ijkt, s}\wedge \theta^s-R_{ijkt, l}\theta^l_s\wedge\theta^s &= -d R_{ljkt}\wedge\theta^l_i+R_{ljkt}\pa{\theta^l_s\wedge\theta^s_i-\Theta^l_i}-d R_{ilkt}\wedge\theta^l_j+R_{ilkt}\pa{\theta^l_s\wedge\theta^s_j-\Theta^l_j}\\ \nonumber &-d R_{ijlt}\wedge\theta^l_i+R_{ijlt}\pa{\theta^l_s\wedge\theta^s_k-\Theta^l_k}-d R_{ijkl}\wedge\theta^l_i+R_{ijkl}\pa{\theta^l_s\wedge\theta^s_t-\Theta^l_t}.
 \end{align}
 Now we repeatedly use \eqref{secondCovDerivRiemComp} and \eqref{1_def_forme_curvatura} into the previous relation; after some manipulations we arrive at
 \begin{align*}
   \pa{d R_{ijkt, s}-R_{ljkt, s}\theta^l_i-R_{ilkt, s}\theta^l_j-R_{ijlt, s}\theta^l_k-R_{ijkl, s}\theta^l_t-R_{ijkt, l}\theta^l_s}\wedge\theta^s &= -\frac 12\left(R_{ljkt}R_{lirs}+R_{ilkt}R_{ljrs}\right. \\ \nonumber &\left.+R_{ijlt}R_{lkrs}+R_{ijkl}R_{ltrs}\right)\theta^r\wedge\theta^s.
 \end{align*}
 Renaming indexes and skew-symmetrizing the left hand side, which is precisely $R_{ijkt, sr}\theta^r\wedge\theta^s$, we obtain \eqref{SecondDerivRiem}.
\end{proof}

As a consequence for the Ricci tensor we have
\begin{lemma}\label{LemmaFSTDerivRicci}
\begin{align}
  &R_{ij, k}-R_{ik, j} = -R_{tijk, t} = R_{tikj, t}; \\ &R_{ij, kt}-R_{ij, tk}=R_{likt}R_{lj}+R_{ljkt}R_{li}.
\end{align}
\end{lemma}
\begin{proof}
  The previous relations follow tracing equations \eqref{SecondBianchiRiem} and \eqref{SecondDerivRiem}, respectively.
\end{proof}

The First Bianchi Identities  imply that
\begin{equation}\label{PermutCiclCotton}
  C_{ijk}+C_{jki}+C_{kij}=0.
\end{equation}
From the definition of the Cotton tensor we also deduce that
\begin{equation}
  C_{ijk, t} = A_{ij,kt}-A_{ik, jt}=R_{ij, kt}-R_{ik,
  jt}-\frac{1}{2(m-1)}\pa{S_{kt}\delta_{ij}-S_{jt}\delta_{ik}}.
\end{equation}
On the other hand, by Lemma \ref{LemmaFSTDerivRicci} and Schur's
identity $S_i = \frac 12 R_{ik, k}$,
\begin{equation}
  R_{ik, jk} = R_{ik, kj}+R_{tijk}R_{tk}+R_{tkjk}R_{ti}=\frac 12
  S_{ij}-R_{tk}R_{itjk}+R_{it}R_{tj}.
\end{equation}
This enables us obtain the following expression for the divergence
of the Cotton tensor:
\begin{equation}\label{DiverCotton}
  C_{ijk, k}= R_{ij, kk}-\frac{m-2}{2(m-1)} S_{ij}+R_{tk}R_{itjk}-R_{it}R_{tj}-\frac{1}{2(m-1)}\Delta S\delta_{ij}.
\end{equation}
The previous relation also shows that
\begin{equation}\label{SymmDivCotton}
C_{ijk, k}=C_{jik, k},
\end{equation}
thus confirming the symmetry of the Bach tensor, see \eqref{def_Bach_comp}.

Taking the covariant derivative of \eqref{PermutCiclCotton} and
using \eqref{SymmDivCotton} we  also deduce
\begin{equation}\label{NullDiverCotton}
  C_{kij, k}=0.
\end{equation}

\

\section{Gradient Einstein-type manifolds with $\alpha=0$}
\label{sec_4}


In this section we will prove Theorem \ref{TH_alfa0_quasiYamabe} and Corollary \ref{COR_alfa0_quasiYamabe} focusing our attention on gradient Einstein-type manifolds with $\alpha=0$. Without loss of generality, we can write the equation in the form
\begin{equation}\label{a}
\hess\pa{f} +\mu \,df\otimes df\, = \, \varphi\, g \,,
\end{equation}
for some $\mu\in\mathds{R}$ and some function  $\varphi\in \cinf$.
Tracing this equation with the metric $g$, we see immediately that
the function $\varphi$ coincides with $(\Delta f+\mu |\nabla
f|^{2})/m$. We prove the following result, which immediately implies Theorem \ref{TH_alfa0_quasiYamabe} and Corollary \ref{COR_alfa0_quasiYamabe}.

\begin{theorem}\label{teo1} Let $\varrg$ be a complete gradient Einstein-type manifold of dimension $m\geq 3$ and of the form~\eqref{a}. Then, any regular level set $\Sigma$ of $f$ admits a maximal open neighborhood $U\subset M^m$ on which $f$ only depends on the signed distance $r$ to the hypersurface $\Sigma$. In addition, the potential function $f$ can be chosen in such a way that the metric $g$ takes the form
\begin{equation}\label{warped metric}
g \, = \, dr \otimes dr \,+ \,\left(\frac{f'(r)}{f'(0)}e^{\mu f(r)}\right)^{2}\, g^{\Sigma} \quad {\hbox{on $U$}} ,
\end{equation}
where $g^{\Sigma}$ is the metric induced by $g$ on $\Sigma$. As a consequence, $f$ has at most two critical points on $M^m$ and we have the following cases:
\begin{itemize}
\item[(1)] If $f$ has no critical points, then $\varrg$ is globally conformally equivalent to a direct product $I\times N^{m-1}$ of some interval $I=(t_{*},t^{*})\subseteq \mathbb{R}$ with a $(m-1)$-dimensional complete Riemannian manifold $(N^{m-1},g^{N})$. More precisely, the metric takes the form
$$
g \, = \, u^{2}(t)\, \big(dt^{2}+g^{N}\big) \, ,
$$
where $u:(t_{*},t^{*})\rightarrow \mathbb{R}$ is some positive smooth function.
\smallskip
\item[(1')] If, in addition, the Ricci tensor of $\varrg$ is nonnegative, then $\varrg$ is {\em isometric} to a
direct product $\mathbb{R}\times N^{m-1}$, where $(N^{m-1},g^N)$ has nonnegative Ricci tensor.
\smallskip

\item[(2)] If $f$ has only one critical point $O\in M^m$, then $\varrg$ is globally conformally equivalent to the interior of a Euclidean ball of radius $t^{*}\in(0,+\infty]$. More precisely, on $M^{m}\setminus~\{O\}$, the metric takes the form
$$
g \, = \, v^{2}(t)\, \big(dt^{2}+t^{2}g^{\mathbb{S}^{m-1}}\big) \,,
$$
where $v:(0,t^{*})\rightarrow \mathbb{R}$ is some positive smooth function. In particular $\varrg$ is complete, noncompact and rotationally symmetric.
\smallskip
\item[(2')] If, in addition, the Ricci tensor of $\varrg$ is nonnegative, then $\varrg$ is globally conformally equivalent to $\mathbb{R}^{m}$.

\smallskip

\item[(3)] If the function $f$ has two critical points $N,S \in M^m$, then $\varrg$ is globally conformally equivalent to $\mathbb{S}^{m}$.  More precisely, on $M^{m}\setminus \{N,S\}$, the metric takes the form
$$
g \, = \, w^{2}(t)\, \big(dt^{2}+\sin^{2}(t)\,g^{\mathbb{S}^{m-1}}\big) \,,
$$
where $w:(0,\pi)\rightarrow \mathbb{R}$ is some smooth positive function. In particular $\varrg$ is compact and rotationally symmetric.
\end{itemize}
\end{theorem}
\begin{proof}
We will follow the proof in~\cite{catmantmazz}, using the Koszul
formalism. Let $\Sigma$ be a regular level set of the function
$f:M^m\to\mathds{R}$, i.e. $|\nabla f|\neq 0$ on $\Sigma$, which
exists by Sard's Theorem and the fact that $f$ is nonconstant in our
definition. First we observe that $|\nabla f|$ has to be constant on
$\Sigma$. Indeed, for all $Y\in T_{p}\Sigma$,
$$
\nabla_Y |\nabla f|^{2} \,=\, 2 \,\hess(f) (\nabla f, Y) =
\frac{2\,\Delta f+2\mu|\nabla f|^{2}}{m} \,g(\nabla f,Y) - 2\mu\,|\nabla f|^{2}\,g(\nabla f,Y) \, = \, 0\,.
$$
From this we deduce that, in a neighborhood $U$ of $\Sigma$ which does
not contain any critical point of $f$, the potential function $f$ only depends
on the signed distance $r$ to the hypersurface $\Sigma$. In particular
$df=f' dr$. Moreover, if $\theta=(\theta^{1}\,\ldots,\theta^{m-1})$ are coordinates {\em adapted} to the hypersurface $\Sigma$, we get
$$
\hess(f) \,=\, \nabla df \,=\, f'' dr\otimes dr + f' \hess(r)=\, f'' dr\otimes dr + \frac{f'}{2} \, \partial_r g_{ij} \,d\theta^i\otimes d\theta^j\,,
$$
since
$$
\Gamma_{rr}^r=\Gamma_{rr}^k=\Gamma_{ir}^r=0\,,\qquad
\Gamma_{ij}^r=- \frac{1}{2} \,\partial_r g_{ij}\,,\qquad
\Gamma_{ir}^k= \frac{1}{2} \, g^{ks}\partial_r g_{is}\,.
$$
On the other hand, using equation \eqref{a}, we have
$$
\hess(f) \, = \, \frac{\Delta f+\mu\,|\nabla f|^{2}}{m}\, g - \mu\, df\otimes df\,= \, \pa{\frac{\Delta
  f+\mu (f')^{2}}{m}-\mu (f')^{2}} \,dr \otimes dr + \pa{\frac{\Delta
  f+\mu (f')^{2}}{m}} g_{ij}\, d\theta^{i} \otimes
d\theta^{j}\,,
$$
thus,
$$
\frac{\Delta f + \mu (f')^{2}}{m}= f''+ \mu (f')^{2} \quad\quad\hbox{and}\quad \quad \frac{\Delta f + \mu (f')^{2}}{m}\, g_{ij} = \frac{1}{2} f' \,\partial_{r}
g_{ij} \,.
$$
These equations imply the family of ODE's
$$
\big[f''(r) +\mu (f')^{2} \big]  \,g_{ij}(r,\theta)  \,=\, \frac{f'(r)}{2} \,\partial_{r}
g_{ij}(r,\theta)\,.
$$
Since $f'(0)\not=0$ (otherwise $\Sigma$ is not a regular level set of $f$) we can integrate these equations obtaining
$$
g_{ij}(r,\theta) \, = \, \left(\frac{f'(r)}{f'(0)}e^{\mu [f(r)-f(0)]}\right)^{2} g_{ij}(0,\theta)\,.
$$
Therefore, in $U$ the metric takes the form
$$
g \, = \, dr \otimes dr \,+ \,\left(\frac{f'(r)}{f'(0)}e^{\mu [f(r)-f(0)]}\right)^{2}\, g^{\Sigma}_{ij}(\theta)\,d\theta^{i} \otimes d\theta^{j}\,,
$$
where $g^{\Sigma}_{ij}(\theta)=g_{ij}(0,\theta)$ is the metric induced by $g$ on $\Sigma$. We notice that, since $f=f(r)$, then the width of the neighborhood $U$ is uniform with respect to the points of $\Sigma$, namely we can assume $U=\{r_{*}<r<r^{*}\}$, for some maximal $r_{*}\in[-\infty,0)$ and $r^{*}\in(0,\infty]$. Moreover, by translating the function $f$, we can assume that $f(0)=0$. Hence, in $U$, the metric can be written as
\begin{equation}\label{metric}
g \, = \, dr \otimes dr \,+ \,\left(\frac{f'(r)}{f'(0)}e^{\mu f(r)}\right)^{2}\, g^{\Sigma}\,,
\end{equation}
where $g^{\Sigma}$ denotes the induced metric on the level set $\Sigma$. Then, if we let
$$
\omega(r):= \frac{f'(r)}{f'(0)}e^{\mu f(r)}\,,
$$
we have that $g= dr \otimes dr \,+ \, \omega(r)^{2} g^{\Sigma}$. At this point, we can follow directly the computations in the proof of~\cite[Theorem 1.1]{catmantmazz}. In fact, one observes that we have three possible cases, depending on the zeros of the function $\omega$, i.e. depending on the number of critical points of the function $f$. Now, to conclude the proof of the theorem one can follow step by step the proof in ~\cite{catmantmazz}.
%
%
\end{proof}


\

\section{The tensor $D$ and the integrability conditions}
\label{sec_5}

The main result of this section concerns two natural integrability
conditions that follow directly from the Einstein-type structure; as
in the case of Ricci solitons and Yamabe (quasi)-solitons, there is a
natural tensor that turns out to play a fundamental role in relating
the Einstein-type structure to the geometry of the underlying
manifold. Quite surprisingly, as it is shown in Theorem
\ref{TH_integrabilityConditions}, the presence of the constant $\rho$ and of the function $\lambda$
seems to be completely irrelevant.

Let $\varrg$ be gradient Einstein-type manifold of dimension $m\geq 3$. Equation \eqref{Eq_ETS_global} in components reads as
\begin{equation}\label{Eq_ETS_components}
  \alpha R_{ij} +\beta f_{ij}+\mu f_if_j = \pa{\rho S+\lambda}\delta_{ij}.
\end{equation}
Tracing the previous relation we immediately deduce that
\begin{equation}\label{Eq_ETS_global_traced}
  \pa{\alpha-m\rho}S +\beta\Delta f +\mu\abs{\nabla f}^2 = m\lambda.
\end{equation}

\begin{defi}
We define the tensor $D$ by its components
\begin{align}\label{Definition_of_D}
   D_{ijk}=\frac{1}{m-2}\pa{f_kR_{ij}-f_jR_{ik}}+\frac{1}{(m-1)(m-2)}f_t\pa{R_{tk}\delta_{ij}-R_{tj}\delta_{ik}}-\frac{S}{(m-1)(m-2)}\pa{f_k\delta_{ij}-f_j \delta_{ik}}.
 \end{align}
  \end{defi}

  Note that $D$ is skew-symmetric in the second and third indices (i.e. $D_{ijk}=-D_{ikj}$) and totally trace-free (i.e. $D_{iik}=D_{iki}=D_{kii}=0$).
 \begin{rem}
   We explicitly note that our conventions for the Cotton tensor and for the tensor $D$ differ from those in \cite{CaoChen}.
 \end{rem}

 \begin{lemma}
 Let $\varrg$ be a gradient Einstein-type manifold of dimension $m\geq 3$. The tensor $D$ can be written in the next three equivalent ways:
    \begin{align}
   D_{ijk} &= \frac{1}{m-2}\pa{f_kR_{ij}-f_jR_{ik}}+\frac{1}{(m-1)(m-2)}f_t\pa{R_{tk}\delta_{ij}-R_{tj}\delta_{ik}}-\frac{S}{(m-1)(m-2)}\pa{f_k\delta_{ij}-f_j \delta_{ik}} \\
   \nonumber &= \frac{1}{m-2}\pa{f_kA_{ij}-f_jA_{ik}} +\frac{1}{(m-1)(m-2)}f_t\pa{E_{tk}\delta_{ij}-E_{tj}\delta_{ik}}\\
   \nonumber  &=\frac{\beta}{\alpha}\sq{\frac{1}{m-2}\pa{f_jf_{ik}-f_kf_{ij}}+\frac{1}{(m-1)(m-2)}f_t\pa{f_{tj}\delta_{ik}-f_{tk}\delta_{ij}}-\frac{\Delta f}{(m-1)(m-2)}\pa{f_j\delta_{ik}-f_k
   \delta_{ij}}},
 \end{align}
where $E_{ij}$ are the components of the \emph{ Einstein tensor}
(see \cite{Besse}) defined as
 \begin{equation*}\label{def_Einstein_comp}
  E_{ij} = R_{ij} - \frac{S}{2}\delta_{ij}.
 \end{equation*}
 \end{lemma}
 Note that the third expression makes sense only if $\alpha\neq 0$. The proof is just a simple computation, using the definitions of the tensors involved, equation \eqref{Eq_ETS_components} and equation \eqref{Eq_ETS_global_traced}.

 The following theorem should be compared with Lemma 3.1 and equation (4.1) in \cite{CaoChen}, with Lemma 2.4 and equation (2.12) in \cite{CaoCatinoChenMantMazz} and with Proposition 2.2 in \cite{HuangLi_quasiYamabe}. This result highlights the geometric relevance of $D$ in this general situation.
\begin{theorem}\label{TH_integrabilityConditions}
Let $\varrg$ be a gradient Einstein-type manifold with $\beta\neq 0$ of dimension
$m\geq 3$.  Then the following integrability conditions hold:
   \begin{align}
     &\alpha C_{ijk}+\beta f_t W_{tijk} = \sq{\beta-\frac{(m-2)\alpha\mu}{\beta}}D_{ijk}, \label{firstGeneralIntCond}\\     &\alpha B_{ij} = \frac{1}{m-2}\set{\sq{\beta-\frac{(m-2)\alpha\mu}{\beta}}D_{ijk, k}+\beta\pa{\frac{m-3}{m-2}}f_tC_{jit}-\mu f_tf_kW_{itjk}}. \label{secondGeneralIntCond}
  \end{align}
\end{theorem}
\begin{proof}
  We begin with the covariant derivative of equation \eqref{Eq_ETS_components} to get
  \begin{equation}
    \alpha R_{ij, k} +\beta f_{ij, k}+\mu \pa{f_{ik}f_j+f_if_{jk}} = \pa{\rho S_k+\lambda_k}\delta_{ij}.
  \end{equation}
  Skew-symmetrizing with respect to $j$ and $k$ and using \eqref{ThirdDerivFunctionRiem} we obtain
  \begin{equation}\label{EQ_FirstCondSkewSymm}
  \alpha\pa{R_{ij, k}-R_{ik, j}}+\beta f_tR_{tijk}+\mu\pa{f_{ik}f_j-f_{ij}f_k}=\rho\pa{S_k\delta_{ij}-S_j\delta_{ik}}+\pa{\lambda_k\delta_{ij}-\lambda_j\delta_{ik}}.
  \end{equation}
  To get rid of the two terms on the right-hand side of equation \eqref{EQ_FirstCondSkewSymm} we proceed as follows:
  first we trace the equation with respect to $i$ and $j$ and we use  Schur's identity $S_k=2R_{tk,
  t}$ to deduce
  \begin{equation}\label{EQ_SkFirst}
  \sq{\alpha-2\rho(m-1)}S_k = 2\beta f_tR_{tk}+2(m-1)\lambda_k-2\mu\pa{f_tf_{tk}-\Delta f f_k};
  \end{equation}
  secondly, from equations \eqref{Eq_ETS_components} and \eqref{Eq_ETS_global_traced} we respectively have
  \begin{equation}\label{EQ_ftk}
  f_{tk}=\frac{1}{\beta}\sq{\pa{\rho S+\lambda}\delta_{tk}-\alpha R_{tk}-\mu f_tf_k}
  \end{equation}
  and
  \begin{equation}
  \Delta f = \frac{1}{\beta}\sq{\pa{m\rho-\alpha}S+m\lambda-\mu\abs{\nabla f}^2}.
  \end{equation}
Inserting the two previous relations in \eqref{EQ_SkFirst} and simplifying we deduce the following important equation

\begin{equation}\label{EQ_Sk}
 \sq{\alpha-2\rho(m-1)}S_k = 2\pa{\beta+\frac{\alpha\mu}{\beta}} f_tR_{tk}+2(m-1)\lambda_k-\frac{2\mu}{\beta}\sq{\alpha-\rho(m-1)}Sf_k+\frac{2\mu}{\beta}(m-1)\lambda f_k.
\end{equation}
From \eqref{Riemann_Weyl} we deduce that
  \begin{equation}\label{EQ_RiemWeylD}
    f_tR_{tijk} = f_tW_{tijk} - D_{ijk}-\frac{1}{m-1}\pa{f_tR_{tk}\delta_{ij}-f_tR_{tj}\delta_{ik}} .
  \end{equation}
  Inserting now \eqref{EQ_RiemWeylD}, \eqref{def_Cotton_comp} and \eqref{EQ_Sk} into \eqref{EQ_FirstCondSkewSymm} and simplifying we get \eqref{firstGeneralIntCond}.

  Taking the divergence of equation \eqref{firstGeneralIntCond} we obtain
  \begin{equation}
  \alpha C_{ijk, k} -\beta f_{tk}W_{itjk} -\beta\pa{\frac{m-3}{m-2}}f_tC_{jit} = \sq{\beta-\frac{(m-2)\alpha\mu}{\beta}}D_{ijk, k};
  \end{equation}
  using the definition of the Bach tensor \eqref{def_Bach_comp}, equation \eqref{EQ_ftk} and the symmetries of $W$ we immediately deduce \eqref{secondGeneralIntCond}.
\end{proof}
\begin{rem}
  Equation \eqref{EQ_Sk} is the analogue of the fundamental $S_k = 2f_tR_{tk}$, valid for every gradient Ricci
  soliton.
\end{rem}

\begin{rem} \label{RM_betanullo} In case $\beta=0$ (and thus $\alpha\neq 0$), by direct calculations, using \eqref{def_Cotton_comp}, \eqref{Definition_of_D} and \eqref{Eq_ETS_components}, one can show that $D = 0$ and equations \eqref{firstGeneralIntCond} and \eqref{secondGeneralIntCond} take the form
$$
\alpha C_{ijk} = -\mu \pa{f_jf_{ik}-f_kf_{ij}}-\frac{\mu}{m-1}f_t\pa{f_{tj}\delta_{ik}-f_{tk}\delta_{ij}}+\frac{\mu \,\Delta f}{m-1}\pa{f_j\delta_{ik}-f_k \delta_{ij}} ,
$$
$$
\alpha B_{ij} = \frac{1}{m-2}\set{\alpha C_{ijk,k}-\mu f_t f_k W_{itjk}} .
$$
\end{rem}

\

\section{Vanishing of the tensor $D$}
\label{sec_6}

In this section we compute the squared norm of the tensor $D$ in terms of $D$ itself, the Bach tensor $B$ and the potential function $f$.  Moreover, under the assumption of Theorem \ref{TH_Main}, we prove the vanishing of $D$. We begin with

%
%
%
%

\begin{lemma} \label{LMasd}
Let $\varrg$ be a nondegenerate gradient Einstein-type manifold of dimension $m\geq 3$. If $\alpha \neq 0$,
  \begin{equation}\label{EQ_Dsquared_general}
    \pa{\frac{m-2}{2}}\sq{\beta-\frac{(m-2)\alpha\mu}{\beta}}\abs{D}^2 = -\beta(m-2)f_if_jB_{ij}+\frac{\beta}{\alpha}\sq{\beta-\frac{(m-2)\alpha\mu}{\beta}}\pa{f_if_jD_{ijk}}_k,
  \end{equation}
 while if  $\alpha=0$
  \begin{equation}\label{EQ_Dsquared_alpha0}
    \pa{\frac{m-2}{2}}\abs{D}^2 = -(m-2)f_if_jB_{ij}+\pa{f_if_jC_{ijk}}_k.
  \end{equation}
 \end{lemma}

\begin{proof}
We observe that, since $D_{ijk}=-D_{ikj}$,
\[
\abs{D}^2 = D_{ijk}D_{ijk} = \frac{1}{m-2}D_{ijk}\pa{f_kR_{ij}-f_jR_{ik}} = \frac{1}{m-2}\pa{f_kR_{ij}D_{ijk}+f_jR_{ik}D_{ikj}},
\] so that
\begin{equation}\label{DSquaredNorm}
\abs{D}^2 = \frac{2}{m-2}f_kR_{ij}D_{ijk}.
\end{equation}

The nondegeneracy condition $\beta-\frac{(m-2)\alpha\mu}{\beta}\neq 0$ implies that, using \eqref{firstGeneralIntCond} and the definition of the Bach tensor, we can write

\begin{align*}
  \pa{\frac{m-2}{2}}\sq{\beta-\frac{(m-2)\alpha\mu}{\beta}}\abs{D}^2 &= f_kR_{ij}\pa{\alpha C_{ijk}+\beta f_t W_{tijk}} \\ &=\alpha f_kR_{ij}C_{ijk}-\beta f_if_jR_{tk}W_{itjk} \\ &=\alpha f_kR_{ij}C_{ijk}-\beta(m-2)f_if_jB_{ij}+\beta f_if_jC_{ijk, k}.
\end{align*}
By the symmetries of the Cotton tensor we also have
\begin{align*}
  f_if_jC_{ijk, k} &= f_i\pa{f_jC_{ijk}}_k -f_if_{jk}C_{ijk} \\ &= \pa{f_if_jC_{ijk}}_k-f_{ik}f_jC_{ijk} \\&=\pa{f_if_jC_{ijk}}_k+f_{ij}f_kC_{ijk},
\end{align*}
therefore we obtain
\begin{align}\label{NormDsquared_first}
  \pa{\frac{m-2}{2}}\sq{\beta-\frac{(m-2)\alpha\mu}{\beta}}\abs{D}^2 &=\alpha f_kR_{ij}C_{ijk}-\beta(m-2)f_if_jB_{ij}+\beta\pa{f_if_jC_{ijk}}_k+\beta f_{ij}f_kC_{ijk}.
  \end{align}
  If $\alpha=0$, using equation \eqref{Eq_ETS_components} in \eqref{NormDsquared_first} we immediately get
  \begin{equation*}
    \pa{\frac{m-2}{2}}\abs{D}^2 =
    -(m-2)f_if_jB_{ij}+\pa{f_if_jC_{ijk}}_k,
  \end{equation*}
  that is \eqref{EQ_Dsquared_alpha0}.

  If $\alpha \neq 0$, using equations \eqref{Eq_ETS_components} and \eqref{firstGeneralIntCond}  in \eqref{NormDsquared_first} and simplifying we deduce
   \begin{equation}
    \pa{\frac{m-2}{2}}\sq{\beta-\frac{(m-2)\alpha\mu}{\beta}}\abs{D}^2 = -\beta(m-2)f_if_jB_{ij}+\frac{\beta}{\alpha}\sq{\beta-\frac{(m-2)\alpha\mu}{\beta}}\pa{f_if_jD_{ijk}}_k,
  \end{equation}
  that is, equation \eqref{EQ_Dsquared_general}.
  \end{proof}
\begin{rem}
In case $\alpha \neq 0$ equation \eqref{EQ_Dsquared_general} can be
obtained in a direct way: one takes the second integrability
condition \eqref{secondGeneralIntCond}, multiplies both members by
$f_if_j$ and simplifies, using the symmetries of the tensors
involved and equation \eqref{firstGeneralIntCond}.
\end{rem}

\begin{theorem}\label{TH_BZDZ}
  Let $\varrg$ be a complete nondegenerate gradient Einstein-type manifold of dimension $m\geq 3$.
  If $B\pa{\nabla f, \cdot} =0$ and $f$ is proper, then $D=0$.
\end{theorem}

\begin{proof}
 We define the vector field $Y=Y(\alpha)$ of components
 \begin{equation}\label{Def_Y}
   Y_k = \begin{cases}
     \frac{\beta}{\alpha}f_if_jD_{ijk} &\text{if} \,\, \alpha \neq 0; \\ f_if_jC_{ijk} &\text{if} \,\, \alpha =0.
   \end{cases}
 \end{equation}
 By the symmetries of $D$ and $C$ we immediately have
 \begin{equation}\label{EQ_YorthogtoNablaf}
 g\pa{Y, \nabla f}=0.
 \end{equation}

  If $B\pa{\nabla f, \cdot} =0$ and $\alpha \neq 0$, from equation \eqref{EQ_Dsquared_general} we obtain
  \begin{equation}
    \pa{\frac{m-2}{2}}\abs{D}^2 =\frac{\beta}{\alpha}\pa{f_if_jD_{ijk}}_k,
  \end{equation}
  while if $\alpha=0$ from equation \eqref{EQ_Dsquared_alpha0} we deduce
  \begin{equation}
     \pa{\frac{m-2}{2}}\abs{D}^2 = \pa{f_if_jC_{ijk}}_k.
  \end{equation}
  In both cases
  \begin{equation}\label{EQ_NormDandDiverY}
     \pa{\frac{m-2}{2}}\abs{D}^2 = \diver Y.
  \end{equation}
  Let now $c$ be a regular value of $f$ and $\Omega_c$ and $\Sigma_c$ be, respectively,  the corresponding sublevel set and level hypersurface, i.e. $\Omega_c = \set{x \in M : f(x)\leq c}$, $\Sigma_c = \set{x \in M : f(x)= c}$.
  Integrating equation \eqref{EQ_NormDandDiverY} on $\Omega_c$ and  using the divergence theorem we get
  \[
  \int_{\Omega_c} \pa{\frac{m-2}{2}}\abs{D}^2 =  \int_{\Omega_c}\diver Y =  \int_{\Sigma_c}g\pa{Y, \nu},
  \]
  where $\nu$ is the unit normal to $\Sigma_c$. Since $\nu$ is in the direction of $\nabla f$, using \eqref{EQ_YorthogtoNablaf} and letting $c\ra +\infty$ we immediately deduce
  \begin{equation}
    \int_M \pa{\frac{m-2}{2}}\abs{D}^2 =0,
  \end{equation}
  which implies $D=0$ on $M$.
\end{proof}
\begin{rem}
The validity of Theorem \ref{TH_BZDZ} is based on that of the divergence theorem in this situation. Thus, instead of using properness of $f$, we can use Theorem A of \cite{GT_ppar} to obtain the above conclusion, that is $D\equiv 0$, under the following assumptions: for some $p>1$, $M$ is $p$-parabolic and the vector field $Y \in L^q(M)$, where $q$ is the conjugate exponent of $p$. We note that a sufficient condition for $p$-parabolicity is
\[
\frac{1}{\operatorname{vol}\pa{\partial B_r}^\frac{1}{p-1}} \not\in L^1\pa{+\infty}
\]
(see e.g. \cite{Tp}), and, according to  \eqref{Def_Y}, $Y\in L^q(M)$ in case for some pair of conjugate exponents $P, P'$ we have
\[
\abs{\nabla f} \in L^{2Pq}\pa{M}\qquad \text{and } \,\, \abs{D} \in L^{P'q}\pa{M} \,\,\text{ if }\, \alpha\neq 0
\]
or
\[
\abs{\nabla f} \in L^{2Pq}\pa{M}\qquad \text{and } \,\, \abs{C} \in L^{P'q}\pa{M} \,\,\text{ if }\, \alpha=0.
\]
\end{rem}

\begin{rem}
A simple computation using  the definition of the tensor $D$ gives
\begin{equation}\label{f_iD_ijk}
f_iD_{ijk} = \frac{1}{m-1}\pa{f_tf_kR_{tj}-f_tf_jR_{tk}},
\end{equation}
 and then
\begin{equation}
f_if_jD_{ijk} = \frac{1}{m-1}\pa{\ricc\pa{\nabla f, \nabla f}f_k - \abs{\nabla f}^2 f_tR_{tk}}.
\end{equation}

  This shows that, in the case $\alpha \neq 0$, the vector field $Y$ defined in \eqref{Def_Y} can be expressed in the remarkable form
  \begin{equation}
    Y = \frac{\beta}{\alpha(m-1)}\sq{\ricc\pa{\nabla f, \nabla f} \nabla f - \abs{\nabla f}^2\pa{\ricc\pa{\nabla f, \cdot}^\sharp}},
  \end{equation}
  where $\sharp$ denotes the usual musical isomorphism.

  Moreover, in the special case of a  gradient Ricci soliton $\pa{M, g, f, \lambda}$, using the fundamental relation $S_k=2f_tR_{tk}$, the vector field $Y$ can also be written in the equivalent form
\[
Y = \frac{1}{2(m-1)} \sq{g\pa{\nabla S, \nabla f}\nabla f-\abs{\nabla f}^2\nabla S}.
\]
We also observe that
  \[
  g\pa{Y, \nabla f} = 0, \,\,\, g\pa{Y, \nabla S} = \frac{1}{2(m-1)}\sq{g\pa{\nabla S, \nabla f}^2-\abs{\nabla S}^2\abs{\nabla f}^2} \leq 0
  \]
  and that
 \[
 \abs{Y}^2 = \frac{1}{4(m-1)^2}\abs{\nabla f}^2\sq{\abs{\nabla S}^2\abs{\nabla f}^2-g\pa{\nabla S, \nabla f}^2} = -\frac{1}{2(m-1)} \abs{\nabla f}^2g\pa{Y, \nabla S} \geq 0.
 \]

\end{rem}

\begin{rem}
In case $\beta=0$ and $\mu\neq 0$, using Remark \ref{RM_betanullo} and arguing as in Lemma \ref{LMasd}, one can obtain the following identity
$$
\frac{\alpha}{2\mu} |C|^{2} = (m-2) f_{i}f_{j}B_{ij} - (f_{i}f_{j}C_{ijk})_{k} .
$$
Then, following the proof of Theorem \ref{TH_BZDZ}, we obtain
\end{rem}

\begin{proposition} \label{PR_betanullo}
 Let $\varrg$ be a complete nondegenerate gradient Einstein-type manifold of dimension $m\geq 3$ and with $\beta=0$.
  If $B\pa{\nabla f, \cdot} =0$ and $f$ is proper, then $C=0$.
\end{proposition}

\

\section{$D$ and the geometry of the level sets of $f$}
\label{sec_7}

In this section we relate the tensor $D$ to the geometry of the regular level sets of the potential function $f$. Our first result highlights, in the case $\alpha\neq 0$, the link between the squared norm of the tensor $D$ and the second fundamental form of the level sets of $f$. This should be compared with \cite[Proposition 3.1]{CaoChen} and \cite[Lemma 4.1]{HDCaoChen_Steady}. For the case $\alpha=0$ we refer to \cite[Proposition 2.3]{HuangLi_quasiYamabe}.

From now on, we extend our index convention assuming $1 \leq i, j, k, \ldots \leq m$ and
$1\leq a, b, c, \ldots \leq m-1$.

\begin{proposition}\label{PR_Dsquared&H}
Let $\varrg$ be a complete $m$-dimensional ($m\geq 3$) gradient Einstein-type
manifold with $\alpha,\beta\neq 0$. Let $c$ be a regular value of $f$ and
let $\Sigma_c = \set{x\in M | f(x)=c}$ be the corresponding level
hypersurface. For $p\in \Sigma_c$ choose an orthonormal frame such
that $\set{e_1, \ldots, e_{m-1}}$ are tangent to $\Sigma_c$ and $e_m
= \frac{\nabla f}{\abs{\nabla f}}$ (i.e., $\set{e_1, \ldots,
e_{m-1}, e_m}$ is a local first order frame along $f$). Then, in $p$,
the squared norm of the  tensor $D$ can be written as
\begin{equation}\label{EQ_d2mad}
  \abs{D}^2 = \pa{\frac{\beta}{\alpha}}^2\frac{2\abs{\nabla f}^4}{\pa{m-2}^2}\abs{h_{ab}-h\delta_{ab}}^2+\frac{2\abs{\nabla f}^2}{(m-1)(m-2)}R_{am}R_{am},
\end{equation}
where $h_{ab}$ are the coefficients of the second fundamental tensor and $h$ is the mean curvature of $\Sigma_c$.
\end{proposition}
\begin{rem}
Note that $\abs{h_{ab}-h\delta_{ab}}^2$ is the squared norm of the
traceless second fundamental tensor $\Phi$ of components $\Phi_{ab}
= h_{ab}-h\delta_{ab}$.
\end{rem}

\begin{proof}
First of all, we observe that, in the chosen frame, we have
\begin{equation}\label{EQ_compfSigma}
df = f_a\theta^a+f_m\theta^m = \abs{\nabla f}\theta^m,
\end{equation}
since $f_a=0, \, a=1, \ldots, m-1$.

The second fundamental tensor $II$ of the immersion $\Sigma_c \hookrightarrow M$ is
\[
II = h_{ab}\theta^b\otimes\theta^a\otimes \nu,
\]
where the coefficients $h_{ab}=h_{ba}$ are defined as
\begin{equation}
\nabla e_m = \nabla \nu = \theta^a_m\otimes e_a = -\theta^m_a\otimes e_a = -h_{ab}\theta^b\otimes e_a
\end{equation}
(see also \cite{MasRigSet}), so that
\begin{equation}
h_{ab} = g\pa{II\pa{e_a, e_b}, \nu}= -g\pa{\nabla_{e_a}\nu, e_b} =  -\pa{\nabla \nu}^\flat\pa{e_a, e_b}.
\end{equation}
In the present setting we have
\[
\nabla \nu = \frac{1}{\abs{\nabla f}}\nabla\pa{\nabla f} +\nabla\pa{\frac{1}{\abs{\nabla f}}}\otimes \nabla f
\]
and
\[
\pa{\nabla \nu}^\flat = \frac{1}{\abs{\nabla f}}\hess(f) +d\pa{\frac{1}{\abs{\nabla f}}}\otimes df,
\]
thus, using equation \eqref{Eq_ETS_components}, we deduce
\begin{equation}\label{EQ_hab}
h_{ab} = -\frac{1}{\abs{\nabla f}}f_{ab} = \frac{1}{\beta\abs{\nabla f}}\sq{\alpha R_{ab}-\pa{\rho S+\lambda}\delta_{ab}},
\end{equation}

The mean curvature $h$ is defined as $h=\frac{1}{m-1}h_{aa}$; tracing equation \eqref{EQ_hab} we get
\begin{equation}\label{EQ_MeanCurvature}
  h = \frac{1}{\beta\abs{\nabla f}}\sq{\pa{\frac{\alpha}{m-1}-\rho}S - \frac{\alpha}{m-1}R_{mm}-\lambda}.
\end{equation}
Now we compute the squared norm of the traceless second fundamental
tensor $\Phi$:
 \begin{align}\label{EQ_PHI}
 \abs{h_{ab}-h\delta_{ab}}^2 &= \abs{h_{ab}}^2-2h h_{aa} +(m-1)h^2 = \abs{h_{ab}}^2-(m-1)h^2 \\ \nonumber &=\frac{1}{\beta^2\abs{\nabla f}^2}\set{\sq{\alpha R_{ab}-\pa{\rho S+\lambda}\delta_{ab}}^2-(m-1)\sq{\pa{\frac{\alpha}{m-1}-\rho}S - \frac{\alpha}{m-1}R_{mm}-\lambda}^2} \\ \nonumber &=\frac{\alpha^2}{\beta^2\abs{\nabla f}^2}\set{\abs{\ricc}^2-2R_{am}R_{am}-\pa{R_{mm}}^2-\frac{1}{m-1}\sq{S^2-2SR_{mm}+\pa{R_{mm}}^2}}\\\nonumber &=\frac{\alpha^2}{\beta^2\abs{\nabla f}^2}\sq{\abs{\ricc}^2-2R_{am}R_{am}-\frac{m}{m-1}\pa{R_{mm}}^2-\frac{1}{m-1}S^2+\frac{2}{m-1}SR_{mm}}.
 \end{align}
 On the other hand, from the definition of $D$ we have
 \begin{align}\label{EQ_Dsq}
 \abs{D}^2 &= \frac{\pa{f_kR_{ij}-f_jR_{ik}}^2}{(m-2)^2}+\frac{\pa{f_tR_{tk}\delta_{ij}-f_tR_{tj}\delta_{ik}}^2}{(m-1)^2(m-2)^2}+\frac{S^2}{(m-1)^2(m-2)^2}\pa{f_k\delta_{ij}-f_j\delta_{ik}}^2\\ \nonumber&+\frac{2}{(m-1)(m-2)^2}\pa{f_kR_{ij}-f_jR_{ik}}\pa{f_tR_{tk}\delta_{ij}-f_tR_{tj}}\\\nonumber&-\frac{2S}{(m-1)(m-2)^2}\pa{f_kR_{ij}-f_jR_{ik}}\pa{f_k\delta_{ij}-f_j\delta_{ik}} \\\nonumber&-\frac{2S}{(m-1)^2(m-2)^2}\pa{f_tR_{tk}\delta_{ij}-f_tR_{tj}\delta_{ik}}\pa{f_k\delta_{ij}-f_j\delta_{ik}} \\\nonumber &=\frac{2\abs{\nabla f}^2}{\pa{m-2}^2}\pa{\abs{\ricc}^2-R_{am}R_{am}-R_{mm}R_{mm}}+\frac{2\abs{\nabla f}^2}{(m-1)(m-2)^2}\pa{R_{am}R_{am}+R_{mm}R_{mm}}\\\nonumber &+\frac{2S^2}{(m-1)(m-2)^2}\abs{\nabla f}^2 + \frac{4\abs{\nabla f}^2}{(m-1)(m-2)^2}\pa{SR_{mm}-\pa{R_{mm}}^2-R_{am}R_{am}} \\ \nonumber &-\frac{4S\abs{\nabla f}^2}{(m-1)(m-2)^2}\pa{S-R_{mm}}-\frac{4S\abs{\nabla f}^2}{(m-1)(m-2)^2}R_{mm}.
 \end{align}
 Symplifying, rearranging and comparing \eqref{EQ_PHI} and \eqref{EQ_Dsq} we arrive at
 \begin{equation}
 \frac{(m-2)^2}{2\abs{\nabla f}^2}\abs{D}^2 = \pa{\frac{\beta}{\alpha}}^2\abs{\nabla f}^2\abs{h_{ab}-h\delta_{ab}}^2+\pa{\frac{m-2}{m-1}}R_{am}R_{am},
 \end{equation}
which easily implies equation \eqref{EQ_d2mad}.

\end{proof}
\begin{rem}
We explicitly note that $R_{am}R_{am}$ is a \emph{globally} defined quantity (since $\abs{D}^2$, $\abs{\nabla f}^2$ and $\abs{h_{ab}-h\delta_{ab}}^2$ are globally defined), but $R_{am}$ is only \emph{locally} defined. This implies that, if $R_{am}=0$ on the open set where the local frame $e_1, \ldots, e_{m}$ is defined, then $d R_{am}=0$ but $R_{am, k}$ is not necessarily zero (see the proof of Proposition \ref{PR_vanishingOfC} below).
\end{rem}

Proposition \ref{PR_Dsquared&H} is one of the key ingredients in the
proof of the following theorem, which generalizes
\cite[Proposition 3.2 ]{CaoChen} (compare also with  in
\cite[Proposition 2.4]{HuangLi_quasiYamabe}). Our proof is similar to those in
\cite{CaoChen} and \cite{HuangLi_quasiYamabe}, but the presence of
$\mu$ and the nonconstancy of $\lambda$ require extra care, in
particular in showing that $S$ is constant on $\Sigma_c$.

\begin{theorem}\label{TH_PropertiesOnSigmac}
Let $\varrg$ be a complete $m$-dimensional, $m\geq 3$, gradient Einstein-type
manifold with $\alpha,\beta \neq 0$ and   tensor $D\equiv 0$. Let $c$ be a
regular value of $f$ and let $\Sigma_c = \set{x\in M | f(x)=c}$ be
the corresponding level hypersurface. Choose any local orthonormal
frame such that $\set{e_1, \ldots, e_{m-1}}$ are tangent to
$\Sigma_c$ and $e_m = \frac{\nabla f}{\abs{\nabla f}}$ (i.e.,
$\set{e_1, \ldots, e_{m-1}, e_m}$ is a first order frame along $f$). Then
\begin{enumerate}
  \item $\abs{\nabla f}^2$ is constant on $\Sigma_c$;
  \item $R_{am}=R_{ma}=0$ for every $a=1, \ldots, m-1$ and $e_m$ is an eigenvector of $\ricc$;
  \item $\Sigma_c$ is totally umbilical;
  \item the mean curvature $h$ is constant on $\Sigma_c$;
  \item the scalar curvature $S$ and $\lambda$ are constant on $\Sigma_c$;
  \item $\Sigma_c$ is Einstein with respect to the induced metric;
  \item on $\Sigma_c$ the (components of the) Ricci tensor of $M$ can be written as $R_{ab} = \frac{S-\Lambda_1}{m-1} \delta_{ab}$, where $\Lambda_1 \in \erre$ is an eigenvalue of multiplicity $1$ or $m$ (and in this latter case $S=m \Lambda_1$); in either case $e_m$ is an eigenvector associated to $\Lambda_1$.
\end{enumerate}
\end{theorem}
\begin{proof}
If $D=0$, from Proposition \ref{PR_Dsquared&H} we immediately deduce that
\begin{equation}\label{EQ_umbilicity}
h_{ab}-h \delta_{ab}=0,
\end{equation}
that is, property (3), and
\begin{equation}\label{EQ_Ram=0}
R_{am}=0, \qquad a=1, \ldots, m-1.
\end{equation}

From \eqref{EQ_umbilicity} a simple computation using \eqref{EQ_hab} and \eqref{EQ_MeanCurvature} shows that
\begin{equation}\label{EQ_RiccRestricted}
R_{ab} = \frac{S-R_{mm}}{m-1}\delta_{ab},
\end{equation}
which also implies
\begin{equation}
  \ricc\pa{\nu, \nu}=\frac{R_{ij}f_if_j}{\abs{\nabla f}^2} = R_{mm} = R_{mm}\abs{\nu}^2;
\end{equation}
this complete the proof of (2).
To prove (1) we take the covariant derivative of $\beta\abs{\nabla f}^2$ and use \eqref{Eq_ETS_components}:
\begin{align*}
\beta\pa{\abs{\nabla f}^2}_k &= 2\beta f_if_{ik} \\&= 2\sq{\pa{\rho S+\lambda-\mu\abs{\nabla f}^2}f_k-\alpha f_tR_{tk}} \\ &= 2\sq{\pa{\rho S+\lambda-\mu\abs{\nabla f}^2}f_k-\alpha f_cR_{ck}-\alpha\abs{\nabla f}R_{mk}};
\end{align*}
evaluating the previous relation at $k=a$ and using property (2) we
immediately get
\[
\pa{\abs{\nabla f}^2}_a=0,
\]
that is (1). To prove (4) we start from Codazzi equations, that in
our setting read
\begin{equation}
  -R_{mabc} = h_{ab, c}-h_{ac, b};
\end{equation}
tracing with respect to $a$ and $c$ we get
\[
-R_{maba} = -R_{mkbk}+R_{mmbm} = h_{ab, a}-h_{aa,b},
\]
that is, using (2),
\begin{equation}\label{EQ_haba}
  0 = -R_{mb} = h_{ab, a}-h_{aa,b}.
\end{equation}
On the other hand, from (3) we have
\[
h_{ab, a} = h_b
\]
and
\[
h_{aa, b} = (m-1)h_b,
\]
so that \eqref{EQ_haba} immediately implies
\begin{equation}
  0= (m-2)h_b, \quad b=1, \ldots, m-1,
\end{equation}
that is (4). To show the validity of (5) we first observe that,
evaluating \eqref{EQ_Sk} at $k=a$ and using (2), we deduce
\[
\sq{\alpha-2\rho(m-1)}S_a-2(m-1)\lambda_a = 0,
\]
which implies
\begin{equation}\label{EQ_SLambdaConst}
  \sq{\alpha-2\rho(m-1)}S-2(m-1)\lambda = \text{const.} \,\,\,\ \text{on}\,\,\Sigma_c.
\end{equation}
From equation \eqref{EQ_MeanCurvature}, the constancy of $h$ and of $\abs{\nabla f}$ on $\Sigma_c$ also give that
\begin{equation}\label{EQ_SLambdaRmmConst}
   \sq{\alpha-\rho(m-1)}S-\alpha R_{mm}-(m-1)\lambda = \text{const.} \,\,\,\ \text{on}\,\,\Sigma_c.
\end{equation}
Combining \eqref{EQ_SLambdaConst} and \eqref{EQ_SLambdaRmmConst} we arrive at
\begin{equation}\label{EQ_S2Rmm}
  S-2R_{mm} = \text{const.} \,\,\,\ \text{on}\,\,\Sigma_c.
\end{equation}
Now we evaluate \eqref{EQ_Sk} at $k=m$, we use (2) and rearrange to
deduce
\begin{align}
  \sq{\alpha-2\rho(m-1)}S_m &= 2\pa{\beta+\frac{\alpha\mu}{\beta}}\abs{\nabla f}R_{mm}+2(m-1)\lambda_m-\frac{2\mu\abs{\nabla f}}{\beta}\set{\sq{\alpha-\rho(m-1)}S-(m-1)\lambda}\\ \nonumber &=2\beta\abs{\nabla f}R_{mm}+2(m-1)\lambda_m-\frac{2\mu\abs{\nabla f}}{\beta}\set{\sq{\alpha-\rho(m-1)}S-\alpha R_{mm}-(m-1)\lambda}.
\end{align}
Since by (1) and \eqref{EQ_SLambdaRmmConst} the quantity
$\frac{2\mu\abs{\nabla f}}{\beta}\set{\sq{\alpha-\rho(m-1)}S-\alpha
R_{mm}-(m-1)\lambda}$ is constant on $\Sigma_c$ we infer
\begin{equation}\label{EQ_SmRmmLambdam}
   \sq{\alpha-2\rho(m-1)}S_m -2\beta\abs{\nabla f}R_{mm}-2(m-1)\lambda_m = \text{const.} \,\,\,\ \text{on}\,\,\Sigma_c.
\end{equation}
Now we take the covariant derivative of \eqref{EQ_SmRmmLambdam} and
evaluate at $k=a$ to obtain
\begin{equation}\label{EQ_SLambdaam}
   \sq{\alpha-2\rho(m-1)}S_{ma} -2\beta\abs{\nabla f}R_{mm, a}-2(m-1)\lambda_{ma} = 0 \,\,\,\ \text{on}\,\,\Sigma_c;
\end{equation}
but $S_{ma}=S_{am}$ and $\lambda_{ma}=\lambda_{am}$, thus \eqref{EQ_SLambdaam} can be written as
\begin{equation}
   \set{\sq{\alpha-2\rho(m-1)}S-2(m-1)\lambda}_{am} = 2\beta\abs{\nabla f}R_{mm, a} \,\,\,\ \text{on}\,\,\Sigma_c,
\end{equation}
which implies, by \eqref{EQ_SLambdaConst}, that
\begin{equation}
  R_{mm} = \text{const.} \,\,\,\ \text{on}\,\,\Sigma_c.
\end{equation}
The previous relation, \eqref{EQ_S2Rmm} and \eqref{EQ_SLambdaConst} show that $S$ and $\lambda$ are constant on $\Sigma_c$, that is (5).
To prove (6) we start from the Gauss equations
\[
{}^{\Sigma_c}R_{abcd} = R_{abcd} +h_{ac}h_{bd}-h_{ad}h_{bc},
\]
which by property (3) can be rewritten as
\begin{equation}\label{EQ_GaussSigma}
  {}^{\Sigma_c}R_{abcd} = R_{abcd} +h^2\pa{\delta_{ac}\delta_{bd}-\delta_{ad}\delta_{bc}}.
\end{equation}
Tracing equation \eqref{EQ_GaussSigma} with respect to $b$ and $d$ gives
\begin{equation}\label{EQ_GaussSigmatraced}
  {}^{\Sigma_c}R_{ac} = R_{ac}-R_{amcm} +(m-2)h^2\delta_{ac};
\end{equation}
tracing again we deduce
\begin{equation}\label{EQ_GaussSigmatracedagain}
  {}^{\Sigma_c}S = S-2R_{mm} +(m-1)(m-2)h^2  = \text{const.} \,\,\,\ \text{on}\,\,\Sigma_c.
\end{equation}
Now a simple computation using decomposition \eqref{Riemann_Weyl} of
the Riemann tensor, equation \eqref{EQ_RiccRestricted} and the fact
that $W_{amcm}=0$ (see  Proposition \ref{PR_vanishingOfC}) shows
that
\begin{equation}\label{EQ_Ramcm}
  R_{amcm}=\frac{1}{m-1}R_{mm}\delta_{ac}.
\end{equation}
Next, inserting \eqref{EQ_RiccRestricted} and \eqref{EQ_Ramcm} into
\eqref{EQ_GaussSigmatraced}, we get
\begin{equation}\label{EQ_EinsteinSigmac}
   {}^{\Sigma_c}R_{ac} = \sq{\frac{S-2R_{mm}}{m-1}+(m-2)h^2}\delta_{ac},
\end{equation}
which shows the validity of (6). Now (7) is an easy consequence of the other properties.
\end{proof}
The next two results are the analogue of \cite[Lemma 4.2]{CaoChen} and \cite[Lemma 4.3]{CaoChen}, respectively.
\begin{proposition}\label{PR_vanishingOfC}
 Let $\varrg$ be a complete noncompact $m$-dimensional ($m\geq 3$) nondegenerate Einstein-type manifold with $\alpha\neq 0$. If $D=0$ then $C=0$ at all points where $\nabla f\neq 0$.
\end{proposition}
\begin{proof}
We choose a local first order frame along $f$ (so that $f_a=0$, $a=1, \ldots, m-1$ and $f_m=\abs{\nabla f}$).
The vanishing of $D$ implies, by the first integrability condition \eqref{firstGeneralIntCond}, that
\begin{equation*}
  \alpha C_{ijk}+\beta f_tW_{tijk}=0,
\end{equation*}
which implies, since $\alpha \neq 0$,
\begin{equation}\label{EQ_CottonWeylDeq0}
  C_{ijk} =-\frac{\beta}{\alpha}f_tW_{tijk}
\end{equation}
and consequently
\begin{equation}
  f_iC_{ijk}=f_mC_{mjk}=\abs{\nabla f}C_{mjk}=0, \,\, j, k =1, \ldots, m;
\end{equation}
thus
\begin{equation}\label{EQ_Cottonmjk}
  C_{mjk}=0.
\end{equation}
Using (3) and (4) of Theorem \ref{TH_PropertiesOnSigmac} we have
\begin{equation}
  h_{ab, c}=0,
\end{equation}
and from the Codazzi equations we get
\begin{equation}
  -R_{mabc}=h_{ab, c}-h_{ac, b} =0;
\end{equation}
since also $R_{am}=0$ by (2) of Theorem \ref{TH_PropertiesOnSigmac}, from the decomposition \eqref{Riemann_Weyl} we easily deduce
\begin{equation}
  W_{ambc} =0,
\end{equation}
which implies by \eqref{EQ_CottonWeylDeq0} that
\begin{equation}\label{EQ_Cottonabc}
  C_{abc}=0.
\end{equation}
By the symmetries of $C$, to conclude it only remains to show that $C_{abm}=0=C_{amb}$.
First we observe that $R_{am}=0$ implies, by the definition of covariant derivative,
\begin{align*}
  0 &= d R_{am} \\ &=R_{km}\theta^k_a+R_{ak}\theta^k_m+R_{am, k}\theta^k\\ &= R_{bm}\theta^b_a +R_{mm}\theta^m_a +R_{ab}\theta^b_m+R_{am}\theta^m_m +R_{am, k}\theta^k \\&=R_{mm}\theta^m_a +R_{ab}\theta^b_m+R_{am, k}\theta^k,
\end{align*}
so that, using \eqref{EQ_RiccRestricted},
\begin{align}\label{EQ_Ramk1}
  R_{am, k}\theta^k =R_{am, b}\theta^b+R_{am, m}\theta^m&= R_{ab}\theta^m_b-R_{mm}\theta^m_a \\ \nonumber &= \pa{\frac{S-R_{mm}}{m-1}\delta_{ab}}\theta^m_b-R_{mm}\theta^m_a \\ \nonumber &=\pa{\frac{S-mR_{mm}}{m-1}}\theta^m_a.
\end{align}
Now we want to show that $R_{am, m}=0$. To see that we first evaluate equation \eqref{Eq_ETS_components} for $i=a$ and $j=m$, obtaining $f_{am}=0$; then we take the covariant derivative of the same equation:
\begin{equation}
  \alpha R_{ij, k}+\beta f_{ijk}+\mu\pa{f_{ik}f_j+f_if_{jk}} = \pa{\rho S_k+\lambda_k}\delta_{ij},
\end{equation}
which for $i=k=m$, $j=a$ gives (using $f_{am}=0$)
\begin{equation}
  \alpha R_{am, m} = -\beta f_{mam};
\end{equation}
but
\[
f_{mam}=f_{mma}+f_iR_{imam} = f_{mma},
\]
while \eqref{Eq_ETS_global_traced} and Theorem \ref{TH_PropertiesOnSigmac} tell us that the (globally defined) quantity $\Delta f$ is constant on $\Sigma_c$, so that
\begin{equation}
  \pa{\Delta f}_a =0.
\end{equation}
On the other hand, from \eqref{Eq_ETS_components} and \eqref{EQ_RiccRestricted} we deduce
\begin{equation}\label{EQ_fabdelta}
\beta f_{ab} = -\frac{1}{m-1}\set{\sq{\alpha-\rho(m-1)}S-\alpha R_{mm}-(m-1)\lambda}\delta_{ab}-\frac{1}{m-1}f_{cc}\delta_{ab}
\end{equation}
which implies, by tracing,  that
\begin{equation}
  \beta\pa{\Delta f-f_{mm}} = \text{const.} \,\,\,\ \text{on}\,\,\Sigma_c;
\end{equation}
in particular
\begin{equation}
  f_{mam} = f_{mma} =  \pa{\Delta f}_a =0,
\end{equation}
and thus
\begin{equation}
  R_{am, m}=0.
\end{equation}
Getting back to equation \eqref{EQ_Ramk1} we now have
\begin{equation}
R_{am, b}\theta^b=\pa{\frac{S-mR_{mm}}{m-1}}\theta^m_a,
\end{equation}
and thus
\begin{align}\label{EQ_Rambfab}
  R_{am, b}&=\pa{\frac{S-mR_{mm}}{m-1}}\theta^m_a\pa{e_b} \\ \nonumber &=\frac{1}{\abs{\nabla f}}\pa{\frac{mR_{mm}-S}{m-1}}f_{ab}.
\end{align}
Schur's identity implies
\begin{equation}\label{EQ_Schurm}
  S_m = 2R_{im, i} = 2R_{am, a}+2R_{mm, m};
\end{equation}
from the definition of $C$ we have, using \eqref{EQ_RiccRestricted} and \eqref{EQ_Rambfab},
\begin{align}\label{EQ_Cabm1}
  C_{abm} &= R_{ab, m}-R_{am, b} - \frac{1}{2(m-1)}S_m\delta_{ab} \\ \nonumber &=\frac{S_m-R_{mm, m}}{m-1}\delta_{ab}+\frac{1}{\abs{\nabla f}}\pa{\frac{s-mR_{mm}}{m-1}}f_{ab}- \frac{1}{2(m-1)}S_m\delta_{ab}\\ \nonumber &=\frac{1}{2(m-1)}S_m\delta_{ab}-\frac{1}{m-1}R_{mm, m}\delta_{ab}+\frac{1}{\abs{\nabla f}}\pa{\frac{S-mR_{mm}}{m-1}}f_{ab}.
\end{align}
Using \eqref{EQ_Schurm}, \eqref{EQ_Rambfab} and \eqref{EQ_fabdelta} into \eqref{EQ_Cabm1} we arrive at
\begin{align}
  C_{abm} &= \frac{1}{m-1}R_{cm, c}\delta_{ab}+\frac{1}{\abs{\nabla f}}\pa{\frac{S-mR_{mm}}{m-1}}f_{ab}\\ \nonumber &=-\frac{1}{m-1}\frac{1}{\abs{\nabla f}}\pa{S- m R_{mm, m}}f_{ab}+\frac{1}{\abs{\nabla f}}\pa{\frac{S-mR_{mm}}{m-1}}f_{ab} \\ \nonumber &=0,
\end{align}
concluding the proof.
\end{proof}

In dimension four, we can prove the following
\begin{cor}\label{COR_vanishingOfW}
 Let $(M^{4},g)$ be a complete noncompact nondegenerate Einstein-type manifold of dimension four with $\alpha\neq 0$. If $D=0$ then $W=0$ at all points where $\nabla f\neq 0$.
\end{cor}
\begin{proof}
From Proposition \ref{PR_vanishingOfC}, we know that $C_{ijk}=0$. Hence, from \eqref{firstGeneralIntCond}, we deduce $f_{t}W_{tijk}=0$ fora any $i,j,k=1,\ldots,4$. For any $p\in M^{4}$ such that $\nabla f (p) \neq 0$, we choose an orthonormal frame $\set{e_1, \ldots, e_{4}}$ such that $e_4
= \frac{\nabla f}{\abs{\nabla f}}$, thus we have
$$
W_{4ijk}(p) \,=\, 0, \quad\quad \hbox{for} \quad i,j,k=1,\ldots 4 \,.
$$
It remains to show that $W_{abcd}(p)=0$ for any $a,b,c,d=1,2,3$. This follows from the symmetries and the traceless property of the Weyl tensor (for instance, see \cite[Lemma 4.3]{CaoChen}).
\end{proof}

\

\section{Proof of the main theorems and some geometric applications}
\label{sec_8}

In this last section we first prove Theorem \ref{TH_Main} and Corollary \ref{COR_Main}. Then, we give some geometric applications in the special cases of gradient Ricci solitons, $\rho$-Einstein solitons and Ricci almost solitons. We begin with

\

{\em Proof of Theorem \ref{TH_Main}}. From Theorem \ref{TH_BZDZ} we know that the tensor $D$ has to vanish on $M$.  Let $\Sigma$ be a regular level set of the function $f:M^m\to\mathds{R}$, i.e. $|\nabla f|\neq 0$ on $\Sigma$, which exists by Sard's Theorem and the fact that $f$ is nontrivial. By Theorem \ref{TH_PropertiesOnSigmac} (1) we have that $|\nabla f|$ has to be
constant on $\Sigma$. Thus, in a neighborhood $U$ of $\Sigma$ which does not contain any critical point of $f$, the potential function $f$ only depends on the signed distance $r$ to the hypersurface $\Sigma$. Hence, by a suitable change of variable, we can express the metric $g_{ij}$
as
$$
ds^2=dr^2+g_{ab}(r, \theta) d\theta^a \otimes
d\theta^{b}\ , \quad r_*<r<r^*\,,
$$
for some maximal $r_{*}\in[-\infty,0)$ and $r^{*}\in(0,\infty]$, where  $(\theta^2, \cdots, \theta^m)$ is  any local coordinates system on the level surface $\Sigma$. Moreover,  by Theorem \ref{TH_PropertiesOnSigmac} (3)-(4), we have
$$ \frac{\partial } {\partial r} g_{ab}=-2 h_{ab}= \phi (r) g_{ab}\ ,$$ where $\phi (r)=-2h(r)$. Thus, it follows easily that
$$g_{ab} (r, \theta)=e^{\Phi(r)} g_{ab} (0, \theta),$$ where
$$
\Phi(r)=\int_{0}^{r} \phi(r) \, dr.
$$
This proves that on $U$ the metric $g$ takes the form of a warped product metric:
$$
ds^2 = dr^{2} + w(r)^{2} g^{E}\,, \quad r\in(r_*, r^*) \,,
$$
where $w$ is some positive smooth function on $U$, and $ g^{E}=g^{\Sigma}$ is the metric defined on the level surface $\Sigma$, which is Einstein, by Theorem \ref{TH_PropertiesOnSigmac} (6). This concludes the proof of Theorem \ref{TH_Main}.

\

{\em Proof of Corollary \ref{COR_Main}.} The proof of Corollary \ref{COR_Main} follows from all the previous considerations combined with Corollary \ref{COR_vanishingOfW}.
\begin{flushright} $\Box$ \end{flushright}

Next we show that the properness assumption on the potential function $f$ in Theorem~\ref{TH_Main} is automatically satisfied by some classes of Einstein-type manifolds.

First of all, let $\varrg$ be a complete, noncompact, {\em gradient Ricci soliton} with potential function $f$. Then, it is well known that $f$ is always proper, provided that the soliton is either shrinking \cite[Theorem 1.1]{CaoZhou}, or steady with positive Ricci curvature and scalar curvature attaining its maximum at some point \cite[Proposition 2.3]{HDCaoChen_Steady} or expanding with nonnegative Ricci curvature \cite[Lemma 5.5]{CaoCatinoChenMantMazz}. Hence, in these cases, Theorem \ref{TH_Main} provides a local version of the classification results obtained in \cite{CaoChen} and \cite{CaoCatinoChenMantMazz}.

Secondly, if $\varrg$ is a complete, noncompact, {\em gradient shrinking $\rho$-Einstein soliton} with $\rho>0$ and bounded scalar curvature, then it follows by \cite[Lemma 3.2]{CatinoMazzMongodi_rhoEinstein2} that the potential function $f$ is proper. Hence, Theorem \ref{TH_Main} implies the following

\begin{theorem}\label{TH_appl1}
Let $\varrg$ be a complete, noncompact  gradient shrinking $\rho$-Einstein soliton of dimension $m\geq 3$ with bounded scalar curvature and $\rho>0$. If $B\pa{\nabla f, \cdot}=0$, then around any regular point of $f$ the manifold $\varrg$ is locally a warped product with $(m-1)$-dimensional Einstein fibers.
\end{theorem}

Finally, we want to show the following result concerning {\em gradient Ricci almost solitons} which are ``strongly'' shrinking.

\begin{theorem}\label{TH_appl2}
Let $\varrg$ be a complete, noncompact  gradient Ricci almost soliton of dimension $m\geq 3$ with bounded Ricci curvature and with $\lambda\geq \underline{\lambda}>0$, for some $\underline{\lambda}$. If $B\pa{\nabla f, \cdot}=0$, then around any regular point of $f$ the manifold $\varrg$ is locally a warped product with $(m-1)$-dimensional Einstein fibers.
\end{theorem}

\begin{proof} By Theorem \ref{TH_Main} it is sufficient to show that under these assumptions the potential function is proper. To do this we will apply a second variation argument as in \cite[Theorem 1.1]{CaoZhou}. Let $r(x)=\hbox{dist}(x,o)$, for some fixed origin $o\in M$. We will show that, for $r(x) \gg 1$,
$$
f(x) \,\geq\, \frac{1}{2}\, \underline{\lambda} \big( r(x) - c \big) ^{2} \,,
$$
for some positive constant $c>0$ depending only on $m$ and on the geometry of $g$ on the unit ball $B_{o}(1)$. Let $\gamma(s)$, $0\leq s \leq s_{0}$ for some $s_{0}>0$, be any minimizing unit speed geodesic starting from $o=\gamma (0)$ and let $\dot{\gamma}(s)$ be the unit tangent vector of $  \gamma$. Then by the second variation of the arc length, we have
$$
\int_{0}^{s_{0}} \phi^{2}(s) \ricc(\dot{\gamma},\dot{\gamma})\,ds \leq (m-1) \int_{0}^{s_{0}}|\dot{\phi}(s)|^{2}\,ds \,,
$$
for every nonnegative function $\phi:[0,s_{0}]\rightarrow \mathbb{R}$. We choose $\phi(s)=s$ on $[0,1]$, $\phi(s)=1$ on $[1,s_{0}-1]$ and $\phi(s)=s_{0}-s$ on $[s_{0}-1, s_{0}]$. Then, since the solitons has bounded Ricci curvature, one has
$$
\int_{0}^{s_{0}} \ricc(\dot{\gamma},\dot{\gamma})\,ds \,\leq\, 2(m-1) + \max_{B_{1}(o)}|\ricc| + \max_{B_{1}(\gamma(s_{0}))}|\ricc| \,\leq\, C\,,
$$
for some positive constant $C$ independent of $s_0$. On the other hand, from the soliton equation, we have
$$
\nabla_{\dot{\gamma}} \nabla_{\dot{\gamma}} f \,=\, \lambda - \ricc(\dot{\gamma},\dot{\gamma})\,.
$$
Integrating along $\gamma$, we get
$$
\dot{f}\big(\gamma(s_{0})\big) - \dot{f}\big(\gamma(0)\big) \,=\, \int_{0}^{s_{0}} \lambda\,ds - \int_{0}^{s_{0}}\ricc(\dot{\gamma},\dot{\gamma})\,ds \geq \underline{\lambda} \, s_{0} - C \,.
$$
Integrating again, we obtain the desired estimate
$$
f\big(\gamma(s_{0})\big) \,\geq\, \frac{1}{2}  \underline{\lambda} \big( s_{0} - c \big) ^{2} \,,
$$
for some constant $c$. This concludes the proof of the theorem.
\end{proof}

\begin{rem}
  As it is clear from the above proof, in case $\underline{\lambda} =\underline{\lambda}(r)$ is such that $\frac{1}{\underline{\lambda}(r)} = o\pa{\frac{1}{r^2}}$ as $r \ra +\infty$ we have $f(r)\ra +\infty$ as $r \ra +\infty$. This suffices to prove \ref{TH_appl2}.
\end{rem}

To conclude, we note that Ricci almost solitons which are warped product were constructed in \cite[Remark 2.6]{PRRS_Almost}.

\vspace{2cm}

{\bf Acknowledgements.}
The second author would like to thank Francesca Savini for some useful remarks.

\bibliographystyle{plain}

\bibliography{bibConfRicciSol_NOV}
\end{document}